\documentclass[a4paper,10pt]{amsart}
\usepackage{etex}

\usepackage[T1]{fontenc}
\usepackage{lmodern}
\usepackage[utf8]{inputenc}
\usepackage{amssymb}
\usepackage{enumitem}
\usepackage{mathrsfs,dsfont,mathtools}
\usepackage{nicefrac}

\usepackage{xcolor}

\newcommand{\splitt}{even}

\usepackage{tikz}
\usetikzlibrary{matrix,arrows,calc}
\tikzset{cd/.style=matrix of math nodes,row sep=2em,column sep=2em, text height=1.5ex, text depth=0.5ex}
\tikzset{cdar/.style=->,auto}
\tikzset{mid/.style={anchor=mid}} 
\tikzset{dar/.style={double,double equal sign distance,-implies}}
\tikzset{narrowfill/.style={inner sep=1pt, fill=white}}

\usepackage[pdftitle={A more general method to classify up to equivariant KK-equivalence},
  pdfauthor={Rasmus Bentmann, Ralf Meyer},
  pdfsubject={Mathematics}
]{hyperref}
\usepackage[lite]{amsrefs}
\newcommand*{\MRref}[2]{ \href{http://www.ams.org/mathscinet-getitem?mr=#1}{MR \textbf{#1}}}
\newcommand*{\arxiv}[1]{\href{http://www.arxiv.org/abs/#1}{arXiv: #1}}

\usepackage{microtype}

\usepackage[all]{xy}

\numberwithin{equation}{section}
\theoremstyle{plain}
\newtheorem{theorem}[equation]{Theorem}
\newtheorem{lemma}[equation]{Lemma}
\newtheorem{proposition}[equation]{Proposition}
\newtheorem{corollary}[equation]{Corollary}
\theoremstyle{definition}
\newtheorem{definition}[equation]{Definition}

\theoremstyle{remark}
\newtheorem{remark}[equation]{Remark}
\newtheorem{example}[equation]{Example}

\newcommand*{\C}{\mathbb C}
\newcommand*{\Z}{\mathbb Z}

\newcommand*{\R}{\mathbb R}
\newcommand*{\T}{\mathbb T}

\newcommand*{\id}{\textup{id}}

\newcommand*{\KK}{\textup{KK}}
\newcommand*{\XK}{\textup{XK}}
\newcommand*{\K}{\textup{K}}
\newcommand*{\E}{\textup{E}}

\newcommand*{\op}{\textup{op}}
\newcommand*{\ev}{\textup{ev}}
\newcommand*{\loc}{\textup{loc}}
\newcommand*{\Cst}{\textup C^*}
\newcommand*{\Cstar}{\texorpdfstring{$\textup C^*$\nb-}{C*-}}
\newcommand*{\Star}{\texorpdfstring{$^*$\nb-}{*-}}

\newcommand*{\Ab}{\mathfrak{Ab}} 
\newcommand*{\KKcat}{\mathfrak{KK}}

\newcommand*{\CONT}{\textup C} 
\newcommand*{\Abel}{\mathfrak A} 
\newcommand*{\Tri}{\mathfrak T}  
\newcommand*{\Ideal}{\mathfrak I}
\newcommand*{\Proj}{\mathfrak P} 

\newcommand*{\Repr}{\mathcal{R}} 
\newcommand*{\Boot}{\mathcal{B}}

\newcommand*{\nb}{\nobreakdash}  

\newcommand*{\Modc}[1]{\mathfrak{Mod}\bigl(#1\bigr)_\textup{c}}
\newcommand*{\Mod}[1]{\mathfrak{Mod}\bigl(#1\bigr)}
\newcommand*{\blank}{\text{\textvisiblespace}}

\mathtoolsset{mathic}

\DeclarePairedDelimiter{\gen}{\langle}{\rangle}

\DeclareMathOperator{\Ext}{Ext}
\DeclareMathOperator{\Hom}{Hom}

\DeclareMathOperator{\Ind}{Ind}

\DeclareMathOperator{\coker}{coker}
\DeclareMathOperator{\Prim}{Prim}

\newcommand*{\inOb}{\mathrel{\in\in}}
\newcommand*{\defeq}{\mathrel{\vcentcolon=}}

\newcommand*{\into}{\rightarrowtail}
\newcommand*{\prto}{\twoheadrightarrow}
\newcommand*{\lad}{\vdash}

\newdir{ >}{{}*!/-5pt/@{>}}

\begin{document}

\title[Classify up to equivariant KK-equivalence]{A more general method to classify up to equivariant KK-equivalence}

\author{Rasmus Bentmann}
\email{rasmus.bentmann@mathematik.uni-goettingen.de}

\author{Ralf Meyer}
\email{rmeyer2@uni-goettingen.de}

\address{Mathematisches Institut\\
  Georg-August Universit\"at G\"ottingen\\
  Bunsenstra\ss{}e 3--5\\
  37073 G\"ottingen\\
  Germany}

\begin{abstract}
  Using a homological invariant together with an obstruction class in
  a certain \(\Ext^2\)-group, we may classify objects in triangulated
  categories that have projective resolutions of length two.  This
  invariant gives strong classification results for actions of the
  circle group on \(\Cst\)\nb-algebras, \(\Cst\)\nb-algebras over
  finite unique path spaces, and graph \(\Cst\)\nb-algebras with
  finitely many ideals.
\end{abstract}


\thanks{The first author was partially supported by the Danish National Research
Foundation through the Centre for Symmetry and Deformation (DNRF92).}

\maketitle

\section{Introduction}
\label{sec:intro}

The \(\Cst\)\nb-algebra classification program aims at classifying
certain \(\Cst\)\nb-algebras up to isomorphism by suitable invariants.
Such a classification usually has two steps.  First, an isomorphism
between the invariants is lifted to an equivalence in a suitable
equivariant \(\KK\)-theory; then the latter is lifted to an
isomorphism.  These two steps are quite different in nature.  The
first is mainly algebraic topology, the second mainly analysis.  This
article deals with the first step of getting equivariant
\(\KK\)-equivalences from isomorphisms on suitable invariants.

The invariants used previously are homological functors -- variants of
\(\K\)\nb-theory.  There are, however, many situations where no
homological invariant is known that is sufficiently fine to detect
\(\KK\)\nb-equivalences.  This article introduces a more complex
invariant with two layers: the primary invariant is a homological
functor as usual, the secondary is a certain obstruction class, which
lives in an \(\Ext^2\)-group constructed from the primary invariant.
We took this idea from Wolbert~\cite{Wolbert:Classifying_K-modules}.
It goes back further to
Bousfield~\cite{Bousfield:K_local_at_add_prime}.

Our two-layer invariants are complete invariants up to
\(\KK\)\nb-equivalence in several new cases and shed light on previous
classification results for non-simple Cuntz--Krieger algebras and
graph algebras.  We explain how to classify arbitrary objects in the
bootstrap class in~\(\KK^\T\), where~\(\T\) is the circle group, and
in~\(\KK(X)\) for a finite unique path space~\(X\) (see
Definition~\ref{def:unique_path_space}).  The latter result is far
more general than previous ones in \cites{Meyer-Nest:Filtrated_K,
  Bentmann-Koehler:UCT}.  Furthermore, we deduce a classification
theorem up to stable isomorphism for purely infinite
graph \(\Cst\)\nb-algebras with finitely many ideals;
this contains the class of real-rank-zero Cuntz--Krieger algebras
classified by Restorff~\cite{Restorff:Classification}.
Our approach has the additional advantage that the resulting
classification result is strong, that is, every isomorphism on the
level of invariants lifts to an isomorphism of \(\Cst\)\nb-algebras;
this also leads to a classification theorem up to actual isomorphism
for the class of unital graph \(\Cst\)\nb-algebras as above.

Our method is based on homological algebra in triangulated
categories, see \cites{Meyer:Homology_in_KK_II,
  Meyer-Nest:Homology_in_KK}.  This starts with a homological
invariant on a triangulated category, which defines a homological
ideal as its kernel on morphisms.  The general theory gives
projective resolutions, derived functors, and a Universal
Coefficient Theorem for objects with a projective resolution of
length~\(1\); this implies that a certain universal homological
invariant -- in practice, this is often the one we started from --
is a complete invariant up to isomorphism in the triangulated
category.  Here we extend this method to also classify objects with
a projective resolution of length~\(2\): we find that objects in the
triangulated category with a given invariant are in bijection with a
certain \(\Ext^2\)-group computed from the given invariant.  Thus
the invariant together with an \(\Ext^2\)-class gives a complete
invariant.

We make this more concrete by the example of the triangulated
category~\(\KK^\T\) of \(\Cst\)\nb-algebras with a circle action.  Our
theorem classifies objects of the (\(\T\)\nb-equivariant) bootstrap
class in~\(\KK^\T\) up to
\(\KK^\T\)-equivalence.  Here the homological invariant is
\(\T\)\nb-equivariant \(\K\)\nb-theory~\(\K^\T_*\).  This is a functor
from~\(\KK^\T\) to the category of \(\Z/2\)\nb-graded, countable
modules over the commutative ring \(R=\Z[x,x^{-1}]\), the
representation ring of~\(\T\).  Generic \(R\)\nb-modules
have projective resolutions of length two, not one.  Hence there is no
Universal Coefficient Theorem in this case.  Let~\(M\) be a countable
\(\Z/2\)\nb-graded \(R\)\nb-module.  We show
\begin{enumerate}
\item there is a \(\T\)\nb-\(\Cst\)-algebra~\(A\) in the
  bootstrap class with \(\K^\T_*(A)\cong M\);
\item for such~\(A\), there is an invariant
  \(\delta(A)\in\Ext^2_R(M,M)^-\) such that
  \(\delta(A_1)=\delta(A_2)\) if and only if there is a
  \(\KK^\T\)\nb-equivalence \(A_1\to A_2\) inducing the identity map on
  \(M=\K^\T_*(A_i)\); here \(\Ext^2_R(M,M)^-\) denotes the odd part of
  the \(\Z/2\)\nb-graded group \(\Ext^2_R(M,M)\).
\end{enumerate}
In particular, if \(\Ext^2_R(M,M)^-=0\) then~\(M\) lifts uniquely to a
\(\T\)\nb-\(\Cst\)-algebra~\(A\) in the bootstrap class.

The above result does not yet finish the classification of
\(\T\)\nb-\(\Cst\)-algebras~\(A\) in the bootstrap class
because there may be isomorphisms \(A_1\to A_2\) that induce a
non-identity isomorphism \(M\to M\) on
\(M=\K^\T_*(A_1)=\K^\T_*(A_2)\).  The complete invariant takes values
in a category of pairs \((M,\delta)\), where~\(M\) is a countable,
\(\Z/2\)\nb-graded \(R\)\nb-module and \(\delta\in\Ext^2_R(M,M)^-\)
and where a morphism \((M_1,\delta_1) \to (M_2,\delta_2)\) is a
grading-preserving \(R\)\nb-module homomorphism \(f\colon M_1\to M_2\)
with \(\delta_2 f = f \delta_1\) in \(\Ext^2_R(M_1,M_2)^-\).  We show
that isomorphism classes in the bootstrap class
in~\(\KK^\T\) are in bijection with isomorphism classes in this
category of pairs.

Many purely infinite \(\Cst\)\nb-algebras carry a
gauge action by~\(\T\) by construction.  As examples of our
classification, we consider Cuntz--Krieger algebras and some
\(\Cst\)\nb-algebras constructed by Nekrashevych
in~\cite{Nekrashevych:Cstar_selfsimilar}.  In these cases,
\(\Ext^2_R(M,M)^-=0\), so that there is no obstruction class.

The above classification result is very efficient for
\emph{counting} isomorphism classes of objects in the bootstrap
class with a given \(\Z/2\)-graded \(R\)\nb-module~\(M\) as its
equivariant \(\K\)\nb-theory.  It may be hard, however, to compute
the obstruction class in~\(\Ext_R^2(M,M)^-\) for a given
\(\T\)\nb-\(\Cst\)-algebra~\(A\) with \(\K^\T_*(A)\cong M\).  At the
moment, we have no examples of non-equivalent
\(\T\)\nb-\(\Cst\)-algebras that are distinguished only by the
obstruction class.  The authors intend to provide adequate methods
for computing obstruction classes in future work.

Our next application concerns the bootstrap class in~\(\KK(X)\) for a
finite topological \(T_0\)\nb-space~\(X\),
see~\cite{Meyer-Nest:Bootstrap}.  Kirchberg's Classification Theorem
says that an equivalence in \(\KK(X)\) between two strongly purely
infinite, stable, separable, nuclear \(\Cst\)\nb-algebras with
primitive ideal space~\(X\) lifts to a \Star{}isomorphism, so
classification up to \(\KK(X)\)\nb-equivalence already implies
classification theorems up to isomorphism for suitable
\(\Cst\)\nb-algebras.  Previous classification results in \(\KK(X)\)
in~\cites{Meyer-Nest:Filtrated_K, Bentmann-Koehler:UCT} only apply
to very special~\(X\) because projective resolutions of length~\(1\)
are rare.

Invariants with enough projective resolutions of length~\(2\) are
more common.  If~\(X\) is a unique path space, then the
\(\K\)\nb-theories of the ideals corresponding to minimal
neighbourhoods of points in~\(X\) give an invariant with this
property.  This invariant is much smaller than filtrated
\(\K\)\nb-theory.  Since any object of the bootstrap class has a
projective resolution of length~\(2\), our new classification method
applies to arbitrary objects in the bootstrap class of~\(\KK(X)\)
for a unique path space~\(X\).

Even if~\(X\) is not a unique path space, our classification theorem
applies to objects in the bootstrap class in~\(\KK(X)\) that have
projective resolutions of length~\(2\).  We show that this is the case
for graph \(\Cst\)\nb-algebras with finitely many ideals.  Furthermore,
we compute the obstruction class of a graph algebra from the
Pimsner--Voiculescu type sequences that compute the \(\K\)\nb-theory
groups of its ideals.  Hence our complete invariant may be computed
effectively in this case.  We get a strong classification functor
up to stable isomorphism for purely infinite graph
\(\Cst\)\nb-algebras with finitely many ideals; strong classification
means that every isomorphism on the invariants lifts to a stable
isomorphism.  This is the first strong classification result -- even for
the class of purely infinite Cuntz--Krieger algebras -- without the
assumption of a specific ideal structure.  The invariant and its
computation are described in more detail in Section~\ref{sec:graph}.

Our abstract setup should also work in many other situations.  One
of them is connective \(\K\)\nb-theory, regarded as an invariant on
connective \(\E\)\nb-theory.  We refer to Andreas Thom's
thesis~\cite{Thom:Thesis} for details.
See~\cite{Dadarlat-McClure:When_are_two} for applications of
connective \(\K\)\nb-theory to \(\Cst\)\nb-algebras.  Another
instance is Kasparov's \(\KK\)\nb-theory for \(\Cst\)\nb-algebras
over a zero-dimensional compact metrisable space~\(X\).  Here the
\(\K\)\nb-theory of the total algebra has a natural module structure
over the ring of locally constant functions \(C(X,\Z)\).  This ring
has global dimension~\(2\)
by~\cite{Finn-Martinez-McGovern:Global_dimension_of_f-ring}*{Examples
  2.5(b)}.

For \(\Cst\)\nb-algebras over the unit interval and filtrated
\(\K\)\nb-theory as the invariant, the relevant Abelian category has
dimension~\(2\) once again.  So far, we cannot treat this example,
however, because there are not enough projective objects in this
case.

\subsection{Outline}
\label{sec:outline}
The structure of this article is as follows.
Section~\ref{sec:lift_twodim} develops the general theory of
obstruction classes.  Section~\ref{sec:KK_T} applies it to circle
actions on \(\Cst\)\nb-algebras, Section~\ref{sec:KK_X} to
\(\Cst\)\nb-algebras over unique path spaces, and
Section~\ref{sec:graph} to graph \(\Cst\)\nb-algebras; this includes
a return to general triangulated categories in order to compute
obstruction classes for objects of a specific type.

\subsection{Acknowledgement}
\label{sec:acknow}

We thank James Gabe and Rune Johansen for pointing out the
references \cite{Finn-Martinez-McGovern:Global_dimension_of_f-ring}
and~\cite{Boyle:Shift_equivalence_Jordan_from}, respectively;
Eusebio Gardella and N.~Christopher Phillips for discussions
on circle actions with the Rokhlin property, and Gunnar Restorff and
Efren Ruiz for discussions on the classification of graph
\(\Cst\)\nb-algebras and for remarks on an earlier version of the
paper.

\section{Lifting two-dimensional objects}
\label{sec:lift_twodim}

Throughout the article, we will freely use the language of homological
algebra in triangulated categories introduced
in~\cite{Meyer-Nest:Homology_in_KK}.

Let~\(\Tri\) be a triangulated category with at least countable
direct sums (so that idempotent morphisms split).  Let~\(\Ideal\) be
a stable homological ideal that is compatible with countable direct
sums and has enough projective objects.  Let \(F\colon
\Tri\to\Abel\) be the universal \(\Ideal\)\nb-exact stable
homological functor.  By \cite{Meyer-Nest:Homology_in_KK}*{Theorem
  57}, the category~\(\Abel\) has enough projective objects, the
adjoint functor~\(F^\lad\) of~\(F\) is defined on all projective
objects of~\(\Abel\), and \(F\circ F^\lad(A)\cong A\) for all
projective objects~\(A\) of~\(\Abel\).  Let
\(\gen{\Proj_\Ideal}\subseteq\Tri\) denote the localising
subcategory generated by the \(\Ideal\)\nb-projective objects.
\cite{Meyer:Homology_in_KK_II}*{Theorem 3.22} implies that
\(\hat{A}\inOb\gen{\Proj_\Ideal}\) if and only if
\(\Tri(\hat{A},B)=0\) for all \(\Ideal\)\nb-contractible objects
\(B\inOb\Tri\); we write~\(\inOb\) for objects of a category.

\begin{definition}
  \label{def:lifting}
  A \emph{lifting} of \(A\inOb\Abel\) is a pair \((\hat{A},
  \alpha)\) consisting of \(\hat{A}\inOb\Tri\) with
  \(\Tri(\hat{A},B)=0\) for all \(\Ideal\)\nb-contractible
  \(B\inOb\Tri\) and an isomorphism \(\alpha\colon F(\hat{A})
  \xrightarrow{\cong} A\).  An \emph{equivalence} between two
  liftings \((\hat{A}_1,\alpha_1)\), \((\hat{A}_2,\alpha_2)\) is an
  isomorphism \(\varphi\in \Tri(\hat{A}_1,\hat{A}_2)\) with
  \(\alpha_1 = \alpha_2\circ F(\varphi)\).  We often drop~\(\alpha\)
  from the notation and call~\(\hat A\) a lifting of~\(A\).
\end{definition}

If \(A\inOb\Abel\) is projective, then~\(F^\lad(A)\) with the
canonical isomorphism \(F\bigl(F^\lad(A)\bigr)\cong A\) is a
\emph{natural} lifting of~\(A\).

\begin{proposition}
  \label{pro:lift_one-dimensional}
  Let \(A\inOb\Abel\) have cohomological dimension~\(1\).
  Then~\(A\) has a lifting, and any two liftings are equivalent.
\end{proposition}

\begin{proof}
  Let
  \[
  0 \to P_1\xrightarrow{\partial} P_0\to A \to 0
  \]
  be a projective resolution in~\(\Abel\).  Then
  \(F^\lad(\partial)\colon F^\lad(P_1)\to F^\lad(P_0)\) is an arrow
  in~\(\Tri\) with \(F\bigl(F^\lad(\partial)\bigr)\cong \partial\).
  Let~\(\hat{A}\) be
  the cone of~\(F^\lad(\partial)\).  Since~\(\partial\) is monic,
  \(F^\lad(\partial)\) is \(\Ideal\)\nb-monic.  Hence
  \(F\bigl(F^\lad(P_1)\bigr)\into F\bigl(F^\lad(P_0)\bigr)\prto
  F(\hat{A})\) is a short
  exact sequence, proving that \(F(\hat{A})\cong A\).  If~\(B\) is
  \(\Ideal\)\nb-contractible, then \(\Tri(F^\lad(P_j),B)
  \cong \Tri\bigl(P_j,F(B)\bigr)=0\) and hence \(\Tri(\hat{A},B)=0\)
  by the long exact sequence for \(\Tri(\blank,B)\).
  Hence~\(\hat{A}\) is a lifting of~\(A\).

  Let \(\hat{A}_1\) and~\(\hat{A}_2\) be liftings of~\(A\).  This
  includes a choice of isomorphisms \(F(\hat{A}_1)\cong A\) and
  \(F(\hat{A}_2)\cong A\).  The Universal Coefficient Theorem
  \cite{Meyer-Nest:Homology_in_KK}*{Theorem 66} applies to
  \(\Tri(\hat{A}_1,\hat{A}_2)\).  Hence there is
  \(f\in\Tri(\hat{A}_1,\hat{A}_2)\) that lifts the identity map
  on~\(A\) when we identify \(F(\hat{A}_1)\cong A\) and
  \(F(\hat{A}_2)\cong A\).  Since~\(f\) is an
  \(\Ideal\)\nb-equivalence, its cone~\(B\) is
  \(\Ideal\)\nb-contractible.  Thus \(\Tri_*(\hat{A}_i,B)=0\) for
  \(i=1,2\), and this implies \(\Tri_*(B,B)=0\) and hence \(B=0\) by
  the long exact sequence.  Thus~\(f\) is invertible.
\end{proof}

The equivalence between two liftings in
Proposition~\ref{pro:lift_one-dimensional} is not canonical, and
the lifting is not natural, unlike for projective objects.  The
Universal Coefficient Theorem
\cite{Meyer-Nest:Homology_in_KK}*{Theorem 66} only shows that any
arrow \(A_1\to A_2\) in~\(\Abel\) between objects of cohomological
dimension~\(1\) lifts to an arrow in~\(\Tri\).  But this lifting is
only unique up to \(\Ext^1(A_1,A_2[-1])\).  With parity assumptions as
in Section~\ref{sec:even_case}, there is a canonical lifting for any
arrow \(A_1\to A_2\): lift its even and odd parts separately and then
take the direct sum.  This shows that the UCT short exact sequence
splits under parity assumptions.  This splitting is not natural,
however.

Proposition~\ref{pro:lift_one-dimensional} implies that isomorphism
classes of objects in~\(\Abel\) of cohomological dimension~\(1\)
correspond bijectively to isomorphism classes of objects~\(A\)
in~\(\gen{\Proj_\Ideal}\) with \(F(A)\) of cohomological
dimension~\(1\).  This is used in \cites{Koehler:Thesis,
  Meyer-Nest:Filtrated_K, Bentmann-Koehler:UCT} and other
classification results.  It may, however, be very hard to find
computable invariants~\(F\) for which all objects in its image have
cohomological dimension~\(1\).

\begin{lemma}
  \label{lem:lift_length-two}
  Any \(A\inOb\Abel\) of cohomological dimension~\(2\) has a lifting
  in~\(\Tri\).
\end{lemma}

\begin{proof}
  Let
  \begin{equation}
    \label{eq:length-two-resolution}
    0 \to P_2 \xrightarrow{\partial_2} P_1 \xrightarrow{\partial_1} P_0
    \xrightarrow{\partial_0} A \to 0
  \end{equation}
  be an exact chain complex in~\(\Abel\) with projective
  \(P_0\), \(P_1\) and~\(P_2\).  Let
  \[
  \Omega A\defeq \partial_1(P_1) \cong  \coker \partial_2,
  \]
  so that we get short exact sequences
  \[
  P_2 \into P_1 \prto \Omega A,\qquad
  \Omega A\into P_0 \prto A.
  \]
  Since~\(\Omega A\) has a projective resolution of
  length~\(1\), it has a lifting \(D\inOb\Tri\) by
  Proposition~\ref{pro:lift_one-dimensional}.  Let
  \(\hat{P}_0\defeq F^\lad(P_0)\) be the canonical lifting
  of~\(P_0\).

  The Universal Coefficient Theorem
  \cite{Meyer-Nest:Homology_in_KK}*{Theorem 66} gives a short exact
  sequence
  \begin{equation}
    \label{eq:UCT_Cokerd2}
    \Ext^1(\Omega A[1],P_0) \into \Tri(D,\hat{P}_0) \prto
    \Hom(\Omega A,P_0).
  \end{equation}
  Hence the inclusion map \(\Omega A\hookrightarrow P_0\) lifts to
  some \(f\in\Tri(D,\hat{P}_0)\), which is \(\Ideal\)\nb-monic.  The
  mapping cone~\(\hat{A}\) of~\(f\) belongs
  to~\(\gen{\Proj_\Ideal}\) by construction, and has \(F(\hat{A})
  \cong P_0/\Omega A \cong A\) by the short exact
  sequence~\eqref{eq:length-two-resolution}, so it is a lifting
  of~\(A\).
\end{proof}

\cite{Meyer-Nest:Filtrated_K}*{Theorem 4.10} shows that liftings of
objects of cohomological dimension two cannot be unique in general.
We may, however, classify liftings up to equivalence:

\begin{theorem}
  \label{the:classify_liftings_two-dim}
  Let \(A\inOb\Abel\) have cohomological dimension~\(2\).  The set
  of equivalence classes of liftings of~\(A\) is in bijection with
  \(\Ext^2(A,A[-1])\).
\end{theorem}

\begin{proof}
  Fix a length-two projective resolution of~\(A\) as
  in~\eqref{eq:length-two-resolution} and a lifting~\(\hat{A}\)
  of~\(A\), which exists by Lemma~\ref{lem:lift_length-two}.  Let
  \(\hat{P}_0\defeq F^\lad(P_0)\).  The map \(P_0\to A\)
  in~\(\Abel\) is adjoint to a map \(\hat{P}_0\to\hat{A}\)
  in~\(\Tri\).  We may complete this to an exact triangle
  \begin{equation}
    \label{eq:triangle_DP0A}
    D\xrightarrow{\varphi} \hat{P}_0\to\hat{A} \to D[1].
  \end{equation}
  Since \(P_0\to A\) is surjective, the map \(\hat{P}_0\to\hat{A}\)
  is an \(\Ideal\)\nb-epimorphism.  Hence~\(F\)
  maps~\eqref{eq:triangle_DP0A} to a short exact sequence \(F(D)
  \into P_0\prto A\).  Thus~\(D\) is a lifting of \(\Omega
  A\defeq\ker \partial_0\cong \coker \partial_2\).  Since~\(\Omega
  A\) has cohomological dimension~\(1\), its lifting~\(D\) is unique
  up to isomorphism by Proposition~\ref{pro:lift_one-dimensional}.
  The exact triangle~\eqref{eq:triangle_DP0A} shows that~\(\hat{A}\)
  is the cone of~\(\varphi\).  So any other lifting~\(\hat{A}'\)
  must be the cone of some arrow \(\varphi'\colon D\to\hat{P}_0\)
  that induces the inclusion map \(F(D)\to P_0\).  Conversely, if
  \(\varphi'\colon D\to\hat{P}_0\) lifts the inclusion map \(\Omega
  A\to P_0\), then its cone is a lifting of~\(A\).

  Let \(\hat{A}\) and~\(\hat{A}'\) be the liftings associated to
  \(\varphi\) and~\(\varphi'\).  We claim that an isomorphism
  \(\alpha\colon \hat{A}\to\hat{A}'\) that induces the identity map
  on \(F(\hat{A}')\cong A\cong F(\hat{A})\) may be embedded in a
  morphism of triangles
  \begin{equation}
    \label{eq:compare_liftings}
    \begin{tikzpicture}[baseline=(current bounding box.west)]
      \matrix(m)[cd,text height=2ex]{
        D&\hat{P}_0&\hat{A}\\
        D&\hat{P}_0&\hat{A}'.\\
      };
      \begin{scope}[cdar]
        \draw (m-1-1) -- node {\(\varphi\)} (m-1-2);
        \draw (m-2-1) -- node {\(\varphi'\)} (m-2-2);
        \draw (m-1-2) -- node {\(\pi_0\)} (m-1-3);
        \draw (m-2-2) -- node {\(\pi_0'\)} (m-2-3);
        \draw (m-1-3) -- node {\(\alpha\)} (m-2-3);
        \draw (m-1-3) -- node {\(\alpha\)} (m-2-3);
        \draw (m-1-1) -- node[swap] {\(\psi\)} (m-2-1);
      \end{scope}
      \draw[double, double equal sign distance] (m-1-2) -- (m-2-2);
    \end{tikzpicture}
  \end{equation}
  The assumption that~\(\alpha\) induces the identity map on~\(A\)
  means that the right square commutes.  This allows to find an
  arrow \(\psi\colon D\to D\) to give a triangle morphism, by an
  axiom of triangulated categories.  The arrow~\(\psi\) induces the
  identity map on \(F(D)=\Omega A\) because the map \(F(D)\to P_0\)
  is injective.  Thus \(\varphi'\circ\psi=\varphi\) for some
  \(\psi\colon D\to D\) that induces the identity map on~\(F(D)\).

  Conversely, let \(\varphi'\circ\psi=\varphi\) for some
  \(\psi\colon D\to D\) that induces the identity map on~\(F(D)\).
  This means that the left square in~\eqref{eq:compare_liftings}
  commutes.  Embed this square in a triangle morphism to construct
  \(\alpha\colon \hat{A}\to\hat{A}'\).  Since~\(\psi\) induces the
  identity map on~\(F(D)\), it is invertible.  Hence~\(\alpha\) is
  also invertible by the Five Lemma for exact triangles.  Summing
  up, \(\hat{A}\) and~\(\hat{A}'\) are equivalent liftings if and
  only if there is \(\psi\colon D\to D\) with
  \(\varphi'\circ\psi=\varphi\) and \(F(\psi)=\id_{\Omega A}\).

  By the Universal Coefficient Theorem, the possible choices for
  \(\psi-\id_D\) and \(\varphi'-\varphi\) lie in \(\Ext^1(\Omega
  A,\Omega A[-1])\) and \(\Ext^1(\Omega A, P_0[-1])\), respectively.
  The short exact sequence \(\Omega A\into P_0 \prto A\) induces a
  long exact sequence
  \[
  \dotsb \to \Ext^1(\Omega A,\Omega A[-1]) \xrightarrow{j}
  \Ext^1(\Omega A,P_0[-1]) \to
  \Ext^1(\Omega A,A[-1]) \to 0
  \]
  because \(\Ext^2(\Omega A,\blank)=0\).  We claim that the two
  liftings \(\hat{A}\) and~\(\hat{A}'\) are equivalent if and only
  if \(\varphi'-\varphi\) belongs to the image of the map~\(j\).
  
  If \(\hat{A}\) and~\(\hat{A}'\) are equivalent, we can write
  \(\varphi'\circ\psi=\varphi\) as above.  Then \(\varphi'-
  \varphi=\varphi'\circ (\id_D-\psi)\).  Since \(\id_D-\psi\)
  belongs to \(\Ext^1(\Omega A,\Omega A[-1])\subseteq\Tri(D,D)\),
  the naturality of the Universal Coefficient Theorem allows to
  compute the element \(\varphi'\circ
  (\id_D-\psi)\) in \(\Ext^1(\Omega A, P_0[-1])\subseteq\Tri(D,
  \hat P_0)\) as the Ext-product of \((\id_D-\psi)\in
  \Ext^1(\Omega A,\Omega A[-1])\) with the induced homomorphism
  \(F(\varphi')\in\Hom(\Omega A,P_0)\).  Since \(F(\varphi')\) is
  the inclusion map \(\Omega A\hookrightarrow P_0\),
  it takes \((\id_D-\psi)\) to \(j(\id_D-\psi)\).  In particular,
  \(\varphi'-\varphi\) belongs to the image of~\(j\).

  Conversely, if \(\varphi'-\varphi = j(\alpha)\) for some
  \(\alpha\in\Ext^1 (\Omega A, \Omega A[-1])\), we may write
  \(\varphi'-\varphi  = j(\id_D-\psi)\) by setting
  \(\psi = \id_D-\alpha\in\Tri(D,D)\).  Since \(F(\alpha) = 0\) we
  have \(F(\psi) = \id_{\Omega A}\).  Moreover, \(\varphi'\circ\psi
  =\varphi\) holds because \(\varphi'-\varphi = j(\id_D-\psi)=
  \varphi'\circ(\id_D-\psi)=\varphi'-\varphi'\circ\psi\).  Hence
  \(\hat{A}\) and~\(\hat{A}'\) are equivalent.

  It follows that \(\hat{A}\) and~\(\hat{A}'\)
  are equivalent if and only if \(\varphi'-\varphi\) is mapped
  to~\(0\) in \(\Ext^1(\Omega A, A[-1])\), and any element in
  \(\Ext^1(\Omega A, A[-1])\) arises from some~\(\varphi'\).
  Since~\(P_0\) is projective, another long exact sequence implies
  \[
  \Ext^1(\Omega A,A[-1])\cong \Ext^2(A,A[-1]).
  \]
  Thus \(\hat{A}\) and~\(\hat{A}'\) are equivalent if and only if
  \(\varphi'-\varphi\) is mapped to~\(0\) in \(\Ext^2(A,A[-1])\),
  and any element in \(\Ext^2(A,A[-1])\) comes from
  some~\(\varphi'\).

  We claim that the map sending~\(\hat{A}'\) to the image of
  \(\varphi'-\varphi\) in \(\Ext^2(A,A[-1])\) is a bijection from
  the set of equivalence classes of liftings of~\(A\) to
  \(\Ext^2(A,A[-1])\).  Indeed, if \(\hat{A}'_1\) and~\(\hat{A}'_2\)
  are two arbitrary liftings, they are cones of maps \(\varphi'_1,
  \varphi'_2\colon D\to \hat{P}_0\) both inducing the inclusion map
  \(\Omega A\to P_0\) on~\(F\).  They liftings \(\hat{A}'_1\)
  and~\(\hat{A}'_2\) are equivalent if and only if the map
  \begin{equation}
    \label{eq:obstruction_class_map}
    \Tri(D,\hat{P}_0)\supseteq
    \Ext^1(\Omega A,P_0[-1])
    \to \Ext^1(\Omega A,A[-1])
    \to \Ext^2(A,A[-1])
  \end{equation}
  sends \(\varphi'_1-\varphi'_2\mapsto0\).
  Since~\eqref{eq:obstruction_class_map} is a group homomorphism,
  this happens if and only if \(\varphi'_1-\varphi\) and
  \(\varphi'_2-\varphi\) have the same image in \(\Ext^2(A,A[-1])\).
\end{proof}

\begin{corollary}
  \label{cor:lift_two-dim_unique_Ext20}
  If \(A\inOb\Abel\) has cohomological dimension~\(2\) and
  \(\Ext^2(A,A[-1])=0\), then up to equivalence there is a
  unique lifting.
\end{corollary}

Our construction actually shows that the set of equivalence classes
of liftings carries a free and transitive action of the Abelian
group \(\Ext^2(A,A[-1])\).  Once we pick a single element, we thus
get a bijection to \(\Ext^2(A,A[-1])\); but this bijection depends
on the choice of one lifting, namely, the one corresponding to
\(0\in\Ext^2(A,A[-1])\).  For two liftings~\(\hat{A}_1\)
and~\(\hat{A}_2\) associated to classes
\(\delta_1,\delta_2\in\Ext^2(A,A[-1])\), the difference
\(\delta_2-\delta_1\in\Ext^2(A,A[-1])\) is canonically defined.  It
is called the \emph{relative obstruction class}
\(\delta(\hat{A}_2,\hat{A}_1)\).

There is a cohomological spectral sequence associated to the
ideal~\(\Ideal\), called ABC spectral sequence
in~\cite{Meyer:Homology_in_KK_II} after Adams, Brinkmann and
Christensen.  It is discussed in great detail in
\cite{Meyer:Homology_in_KK_II}*{Section~4}.  The relative
obstruction class \(\delta(\hat{A}_2,\hat{A}_1)\) is
related to the boundary map on the second page of the ABC
spectral sequence for \(\Tri(\hat{A}_1,\hat{A}_2)\).  The
\(E_2\)\nb-term in this spectral sequence is
\[
E_2^{p,q} = \Ext^p(A,A[q])
\]
for \(p\ge0\), \(q\in\Z\).  By assumption, \(E_2^{p,q}=0\) for
\(p\neq0,1,2\).  Hence \(E_k^{p,q}=0\) for \(p\neq0,1,2\) and
\(k\geq3\) as well.  Since the boundary map~\(d_k\) on~\(E_k\) has
bidegree \((k,1-k)\), we get \(d_k=0\) for \(k\ge3\), and the only
part of~\(d_2\) that may be non-zero is \(d_2^{0,q}\colon
E_2^{0,q}\to E_2^{2,q-1}\).  Hence
\[
E_\infty^{0,q} = \ker d_2^{0,q},\qquad
E_\infty^{1,q} = E_2^{1,q},\qquad
E_\infty^{2,q} = \coker d_2^{0,q+1}.
\]
As a consequence, \(\varphi\in\Hom(A,A)\) lifts to
\(\Tri(\hat{A}_1,\hat{A}_2)\) if and only if
\(d_2^{0,0}(\varphi)=0\).  In particular, \(\hat{A}_1\)
and~\(\hat{A}_2\) are equivalent liftings if and only if \(\id_A\in
\Hom(A,A)\) lifts to \(\Tri(\hat{A}_1,\hat{A}_2)\), if and only if
\[
d_2^{0,0}(\id_A) \in E_2^{2,-1} = \Ext^2(A,A[-1])
\]
vanishes.  Thus both conditions \(d_2^{0,0}(\id_A)=0\) and
\(\delta(\hat A_1,\hat A_2)=0\) are necessary and sufficient for an
equivalence of liftings.  This suggests the following lemma:

\begin{lemma}
  \(d_2^{0,0}(\id_A)= \delta(\hat{A}_1,\hat{A}_2) =
  -\delta(\hat{A}_2,\hat{A}_1)\).
\end{lemma}

\begin{proof}
  The cohomological spectral sequence for
  \(\Tri(\hat{A}_1,\hat{A}_2)\) in
  \cite{Meyer:Homology_in_KK_II}*{Section~4} is constructed using a
  phantom tower for~\(\hat{A}_1\).  We implicitly already constructed
  such a phantom tower when lifting~\(A\) to~\(\hat{A}_1\).  In the
  notation above, it looks as follows:
  \begin{equation}
    \label{eq:phantom_tower}
    \begin{tikzpicture}[scale=2,baseline=(current bounding box.west)]
      \node (N0) at (0,0) {\(\hat{A}_1\)};
      \node (N1) at (1,0) {\(D\)};
      \node (N2) at (2,0) {\(\hat{P}_2\)};
      \node (N3) at (3,0) {\(0\)};
      \node (N4) at (4,0) {\(\cdots\)};
      \node (P0) at (300:1) {\(\hat{P}_0\)};
      \node (P1) at ($(300:1)+(1,0)$) {\(\hat{P}_1\)};
      \node (P2) at ($(300:1)+(2,0)$) {\(\hat{P}_2\)};
      \node (P3) at ($(300:1)+(3,0)$) {\(0\)};
      \node (P4) at ($(300:1)+(4,0)$) {\(\cdots\)};
      \begin{scope}[cdar]
        \draw (N0) -- node {\(\iota_0\)} (N1);
        \draw (N1) -- node {\(\iota_1\)} (N2);
        \draw (N2) -- (N3);
        \draw (N3) -- (N4);
        \draw (P1) -- node {\(\hat{\partial}_1\)} (P0);
        \draw (P2) -- node {\(\hat{\partial}_2\)} (P1);
        \draw (P3) -- (P2);
        \draw (P4) -- (P3);
        \draw (P0) -- node[anchor=north east, inner sep=0pt] {\(\hat{\partial}_0\)} (N0);
        \draw (P1) -- node[anchor=north east, inner sep=0pt] {\(\pi_1\)} (N1);
        \draw (P3) -- (N3);
        \draw (P4) -- (N4);
        \draw (N1) -- node[anchor=south east, inner sep=0pt] {\(\varphi_1\)} (P0);
        \draw (N2) -- node[anchor=south east, inner sep=0pt] {\(\hat{\partial}_2\)} (P1);
        \draw (N3) -- (P2);
        \draw (N4) -- (P3);
      \end{scope}
      \draw[double, double equal sign distance] (P2) -- (N2);
      \draw[fill=white] ($.5*(N0)+.5*(N1)$) circle (.7pt);
      \draw[fill=white] ($.5*(N1)+.5*(N2)$) circle (.7pt);
    \end{tikzpicture}
  \end{equation}
  The circled arrows have degree~\(-1\).  (The conventions about the
  degrees of the maps in the phantom tower are different
  in~\cite{Meyer:Homology_in_KK_II}.)  Here \(\hat{P}_j\)
  and~\(\hat{\partial}_j\) are the unique liftings of the projective
  objects~\(P_j\) and the boundary maps~\(\partial_j\)
  in~\eqref{eq:length-two-resolution} for \(j=0,1,2\),
  and~\(\varphi_1\) is the map with cone~\(\hat{A}_1\) used in the
  arguments above.  The triangles involving~\(\iota_n\) are
  \(\Ideal\)\nb-exact, and the other triangles commute.  This
  together with the \(\Ideal\)\nb-projectivity of the
  objects~\(\hat{P}_n\) means that~\eqref{eq:phantom_tower} is a
  phantom tower.

  The relevant cohomological spectral sequence is constructed by
  applying the cohomological functor \(\Tri(\blank,\hat{A}_2)\) to
  the phantom tower for~\(\hat{A}_1\) in~\eqref{eq:phantom_tower}.
  The boundary map~\(d_2\) on the second page maps \(E_2^{0,0}\to
  E_2^{2,-1}\), where
  \begin{alignat*}{2}
    E_2^{0,0} &\defeq
    \{ x\in\Tri_0(\hat{P}_0,\hat{A}_2) \mid \hat{\partial}_1^*(x)=0\}
    &&\cong \Hom(A,A),\\
    E_2^{2,-1} &\defeq
    \Tri_{-1}(\hat{P}_2,\hat{A}_2) \bigm/
    \hat{\partial}_2^*\bigl(\Tri_{-1}(\hat{P}_1,\hat{A}_2)\bigr)
    &&\cong \Ext^2(A,A[-1]).
  \end{alignat*}
  We describe how~\(d_2\) acts on \(\id_A\in\Hom(A,A)\) (see
  \cite{Meyer:Homology_in_KK_II}*{Section 4.1}).  Let
  \(\hat{\partial}_0'\in\Tri_0(\hat{P}_0,\hat{A}_2)\) be the unique
  element such that \(F(\hat{\partial}_0')=\partial_0\colon P_0\to
  A\).  We have \(\hat{\partial}_0'\circ\hat{\partial}_1=0\) because
  \(F(\hat{\partial}_0'\circ\hat{\partial}_1)=0\).  The isomorphism
  \(\Hom(A,A) \to E_2^{0,0}\) maps \(\id_A\)
  to~\(\hat{\partial}_0'\in\Tri_0(\hat{P}_0,\hat{A}_2)\).  Since
  \(\hat{\partial}_0'\circ \varphi_1\circ\pi_1=
  \hat{\partial}_0'\circ \hat{\partial}_1 = 0\), there is
  \(\rho\in\Tri_{-1}(\hat{P}_2,\hat{A}_2)\) with
  \(\iota_1^*(\rho)=\hat{\partial}_0'\circ \varphi_1\).  The image
  of~\(\rho\) in~\(E_2^{2,-1}\) is \(d_2(\id_A)\).

  The UCT exact sequence for \(\Tri_*(D,\hat{A}_2)\) is the long
  exact sequence associated to the triangle \(\hat{P}_2\hat{P}_1D\)
  in~\eqref{eq:phantom_tower}.  This exact sequence shows
  that~\(\iota_1^*\) is an isomorphism from~\(E_2^{2,-1}\), the
  cokernel of~\(\hat{\partial}_2^*\), onto
  \[
  \Ext^1(\Omega A, A[-1]) \subseteq \Tri_0(D,\hat{A}_2).
  \]
  The map~\(d_2\) is constructed so that
  \(\iota_1^*\bigl(d_2(\id_A)\bigr)=\hat{\partial}_0'\circ
  \varphi_1\).

  Embed~\(\hat{\partial}_0'\) in an exact triangle
  \[
  D \xrightarrow{\varphi_2} \hat{P}_0
  \xrightarrow{\hat{\partial}_0'} \hat{A}_2
  \to D[1]
  \]
  as in~\eqref{eq:triangle_DP0A}.  Then \(\hat{\partial}_0'\circ
  \varphi_2=0\) and hence \(\iota_1^*\bigl(d_2(\id_A)\bigr) =
  \hat{\partial}_0'\circ (\varphi_1-\varphi_2)\).

  The map \(\varphi_2-\varphi_1\) induces the zero map \(F(D)\to
  F(\hat{P}_0)\) and hence corresponds to an element~\(x\) in
  \(\Ext^1(\Omega A, P_0[-1])\) by the UCT sequence.  By definition,
  the obstruction class \(\delta(\hat{A}_2,\hat{A}_1)\) is the image
  of~\(x\) under the map
  \[
  \Ext^1(\Omega A,P_0[-1])\to \Ext^1(\Omega A,A[-1]) \cong
  \Ext^2(A,A[-1]),
  \]
  where the first map is induced by the projection \(P_0\to A\).  By
  the naturality of the UCT sequence, this maps~\(x\) to
  \(\hat{\partial}_0'\circ (\varphi_2-\varphi_1)\).  Comparing this
  with our computation of \(d_2(\id_A)\) shows that
  \(\delta(\hat{A}_2,\hat{A}_1)=-d_2(\id_A)\).
\end{proof}

Our description of equivalence classes of liftings is not yet a
classification of objects in~\(\gen{\Proj_\Ideal}\) up to isomorphism.
Two objects \(\hat{A}_1,\hat{A}_2\in\gen{\Proj_\Ideal}\) are
isomorphic if and only if there is an isomorphism \(F(\hat{A}_1)\to
F(\hat{A}_2)\) that lifts to \(\Tri(\hat{A}_1,\hat{A}_2)\).  If
\(F(\hat{A}_1)=F(\hat{A}_2)\) and
\(\delta(\hat{A}_1,\hat{A}_2)\neq0\), then the identity map
\(F(\hat{A}_1)\to F(\hat{A}_2)\) does not lift; but there may be
another isomorphism \(F(\hat{A}_1)\cong F(\hat{A}_2)\) that lifts to
\(\Tri(\hat{A}_1,\hat{A}_2)\).  This seems hard to decide given only
\(F(\hat{A}_i)\) and \(\delta(\hat{A}_1,\hat{A}_2)\neq0\).

\subsection{Parity assumptions}
\label{sec:even_case}

We are going to impose an extra assumption on~\(\Abel\) that provides
a \emph{canonical} lifting for each object of~\(\Abel\) of
cohomological dimension~\(2\).  This allows us to understand the
action of automorphisms on obstruction classes and to classify objects
of~\(\gen{\Proj_\Ideal}\) with length-\(2\)-projective resolutions up
to isomorphism.

\begin{definition}
  A stable Abelian category is called \emph{\splitt{}} if
  \(\Abel=\Abel_+\times\Abel_-\) with \(\Abel_+[-1]=\Abel_-\) and
  \(\Abel_-[-1]=\Abel_+\); that is, any object of~\(\Abel\) is a
  direct sum of objects of even and odd parity, and the suspension
  automorphism on~\(\Abel\) shifts parity.
\end{definition}

\begin{example}
  \label{exa:Ztwo_graded_split}
  Let~\(\Abel\) be the category of countable,
  \(\Z/2\)\nb-graded modules over a ring~\(R\).  Then~\(\Abel\)
  is \splitt{}, with~\(\Abel_\pm\) the subcategories of countable
  \(R\)\nb-modules concentrated in even or odd degree,
  respectively.

  If \(\Tri=\KK^G\) for a compact group~\(G\) and~\(\Ideal\) is
  the kernel of morphisms of the functor~\(\K^G_*\), then the
  Abelian approximation of~\(\Tri\) with respect to~\(\Ideal\)
  is the category of countable, \(\Z/2\)\nb-graded modules over
  the representation ring of~\(G\).  Hence~\(\Abel\) is \splitt{}
  in this example.
\end{example}

Assume that~\(\Abel\) is \splitt{}.  Since the two
subcategories~\(\Abel_\pm\) are orthogonal, we have
\(\Ext^2(A_+,A_-)=0\) and \(\Ext^2(A_-,A_+)=0\) for
\(A_+\in\Abel_+\), \(A_-\in\Abel_-\).  Now write
\(A\inOb\Abel\) as \(A\cong A_+\oplus A_-\) with
\(A_\pm\in\Abel_\pm\).  Then \(\Ext^2(A_+,A_+[-1])=0\) because
\(A_+[-1]\in\Abel_-\) and \(\Ext^2(A_-,A_-[-1])=0\) because
\(A_-[-1]\in\Abel_+\).
Corollary~\ref{cor:lift_two-dim_unique_Ext20} shows that there
are unique liftings \(\hat{A}_+\) and~\(\hat{A}_-\) for \(A_+\)
and~\(A_-\) (up to equivalence).  We call \(\hat{A}_0\defeq
\hat{A}_+\oplus\hat{A}_-\inOb\Tri\) the \emph{canonical}
lifting of~\(A\) and let \(\delta(\hat{A})\defeq
\delta(\hat{A},\hat{A}_0)\) for any other lifting.  This
defines a canonical obstruction class in \(\Ext^2(A,A[-1])\)
for all liftings~\(\hat{A}\) of~\(A\).  A simple computation as
in \cite{Wolbert:Classifying_K-modules}*{Proposition 9} shows
that, for \(f\in\Hom(A,B)\), the element
\(d_2^{0,0}(f)\in\Ext^2(A,B[-1])\) is given by the formula
\begin{equation}
    \label{eq:differential_in_terms_of_deltas}
  d_2^{0,0}(f)=\delta(\hat{B})f-f\delta(\hat{A}).
\end{equation}

\begin{definition}
  \label{def:Ext_enriched_invariant}
  Let~\(\Abel\delta\) denote the additive category of
  pairs~\((A,\delta)\) with \(A\inOb\Abel\) and
  \(\delta\in\Ext^2(A,A[-1])\); morphisms from \((A,\delta)\)
  to \((A',\delta')\) in~\(\Abel\delta\) are morphisms~\(f\)
  from~\(A\) to~\(A'\) in~\(\Abel\) which satisfy the
  compatibility condition \(\delta' f=f\delta\).
\end{definition}

There is an additive functor
\[
F\delta\colon\Tri\to\Abel\delta,\qquad
\hat{A}\mapsto\bigl(F(\hat{A}),\delta(\hat{A})\bigr).
\]
The following classification result generalises
\cite{Bousfield:K_local_at_add_prime}*{Theorem 9.1} and
\cite{Wolbert:Classifying_K-modules}*{Theorem 11}.

\begin{theorem}
  \label{thm:Ext_enriched_classification}
  Assume that~\(\Abel\) is \splitt{} and has global
  dimension~\(2\).  Then the functor~\(F\delta\) is full and induces a
  bijection between isomorphism classes of objects~\(\hat{A}\)
  in~\(\gen{\Proj_\Ideal}\) and isomorphism classes of objects
  in the category~\(\Abel\delta\).  Furthermore, every lift of
  an isomorphism in~\(\Abel\delta\) is an isomorphism
  in~\(\Tri\).
\end{theorem}

\begin{proof}
  The last claim in the theorem follows from a standard
  argument: if~\(\hat{A}_1\) and~\(\hat{A}_2\) belong
  to~\(\gen{\Proj_\Ideal}\) and if
  \(f\in\Tri(\hat{A}_1,\hat{A}_2)\) is an
  \(\Ideal\)-equivalence, then the mapping cone~\(C_f\)
  of~\(f\) is both \(\Ideal\)\nb-contractible and
  in~\(\gen{\Proj_\Ideal}\); hence \(\Tri(C_f,B)=0\) for all
  \(\Ideal\)\nb-contractible \(B\inOb\Tri\) and in particular
  \(\Tri(C_f,C_f)=0\), so that \(C_f\cong 0\), that is, \(f\) is
  invertible.

  The proof of Theorem~\ref{the:classify_liftings_two-dim} shows
  that every class in \(\Ext^2(A,A[-1])\) appears as \(\delta(\hat
  A)\) for some lifting~\(\hat{A}\) of~\(A\).  Hence~\(F\delta\) is
  essentially surjective.  Since a morphism \(f\in\Hom(A,B)\) lifts
  to a morphism \(\hat{A}\to\hat{B}\) if (and only if)
  \(d_2^{0,0}(f)=0\), \eqref{eq:differential_in_terms_of_deltas}
  shows that the functor~\(F\delta\) is full.  Hence~\(F\delta\)
  distinguishes isomorphism classes.
\end{proof}

\section{Kasparov theory for circle actions}
\label{sec:KK_T}

Let~\(G\) be a connected compact Lie group with torsion-free
fundamental group.  We will soon specialise to the circle group
\(G=\T\), but some results hold more generally.  Let \(\Tri\defeq
\KK^G\) be the \(G\)\nb-equivariant Kasparov theory.  It has
\(\Cst\)\nb-algebras with a continuous \(G\)\nb-action as objects
and \(\KK^G_0(A,B)\) as arrows from~\(A\) to~\(B\).  Its
triangulated category structure is introduced
in~\cite{Meyer-Nest:BC}.

The representation ring~\(R\) of~\(G\) is naturally isomorphic to
\(\KK^G_0(\C,\C)\).  The \(G\)\nb-equivariant \(\K\)\nb-theory
\(\K^G_*(A) \cong \KK^G_*(\C,A)\) is a module over~\(R\) by exterior
product.  Furthermore, \(\K^G_*(A)\cong \K_*(A\rtimes G)\) is
countable if~\(A\) is separable because it is the \(\K\)\nb-theory of
a separable \(\Cst\)\nb-algebra.  Let~\(\Abel\) be the category of
countable, \(\Z/2\)\nb-graded \(R\)\nb-modules.  Let \(F\defeq
\K^G_*\colon \Tri\to\Abel\) be the equivariant \(\K\)\nb-theory
functor.  Let~\(\Ideal\) be the kernel of~\(F\) on morphisms.  This is
a stable homological ideal by definition.

\cite{Meyer-Nest:Homology_in_KK}*{Theorem 72} says that~\(F\) is the
universal \(\Ideal\)\nb-exact stable homological functor and
that~\(\Ideal\) has enough projective objects.  More precisely, the
adjoint~\(F^\lad\) maps the free rank-one module~\(R\) to~\(\C\) with
trivial \(G\)\nb-action because
\[
\KK_*^G(\C,A)\cong \K^G_*(A) \cong \Abel\bigl(R,\K^G_*(A)\bigr).
\]

An object~\(B\) of \(\KK^G\) is \(\Ideal\)\nb-contractible if and only
if \(\K^G_*(B)=0\).  We have \(\KK^G_*(A,B)=0\) for all
\(\Ideal\)\nb-contractible~\(B\) if and only if~\(A\) belongs
to~\(\gen{\C}\), the localising subcategory of~\(\KK^G\) generated
by~\(\C\); this is the correct analogue of the bootstrap class in this
case.
Liftings for objects of~\(\Abel\) are required to belong
to~\(\gen{\C}\).  The following result is implicit
in~\cite{Meyer-Nest:BC_Coactions}.

\begin{proposition}
  \label{pro:bootstrap_in_KKG}
  Let~\(G\) be a connected compact Lie group such that \(\pi_1(G)\) is
  torsion-free.  Let~\(T\) be a maximal torus in~\(G\).  A
  \(G\)\nb-\(\Cst\)-algebra~\(A\) belongs to the equivariant bootstrap
  class~\(\gen{\C}\) if and only if \(A\rtimes G\) belongs to the
  usual bootstrap class in~\(\KK\), if and only if \(A\rtimes T\)
  belongs to the usual bootstrap class in~\(\KK\).
\end{proposition}

\begin{proof}
  If \(A\inOb\gen{\C}\), then \(A\rtimes G\inOb\gen{\C\rtimes
    G}=\gen{\Cst(G)}=\gen{\C}\) because~\(\Cst(G)\) is a direct sum of
  matrix algebras; similarly, \(A\rtimes T\inOb\gen{\C}\).
  Conversely, assume that \(A\rtimes T\) belongs to the usual
  bootstrap class \(\gen{\C}\subseteq \KK\).  The assumptions on~\(G\)
  imply that \(H^2(G,\T)=0\).  Hence
  \cite{Meyer-Nest:BC_Coactions}*{Proposition 3.3} says that any
  \(G\)\nb-\(\Cst\)\nb-algebra~\(A\) belongs to the localising
  subcategory of~\(\Tri\) generated by \(A\rtimes T\) equipped with
  the trivial \(G\)\nb-action.  Taking the trivial \(G\)\nb-action is
  a triangulated functor \(\mathfrak{t}\colon \KK\to\KK^G\), so
  \(A\inOb \gen{\mathfrak{t}(A\rtimes T)}\) and \(A\rtimes T \inOb
  \gen{\C}\) give \(A\inOb \gen{\mathfrak{t}(\C)}\) as asserted.

  There is a Morita--Rieffel equivalence \(A\rtimes T\sim \bigl(A
  \otimes\CONT(G/T)\bigr)\rtimes G\).
  \cite{Meyer-Nest:BC_Coactions}*{Proposition 2.1} says that
  \(\CONT(G/T)\) is \(\KK^G\)-equivalent to~\(\C^w\), where~\(w\) is
  the size of the Weyl group of~\(G\).  Hence \(A\rtimes T\) is
  \(\KK\)\nb-equivalent to \((A\rtimes G)^w\).  Thus \(A\rtimes T\inOb
  \gen{\C}\) if and only if \(A\rtimes G\inOb \gen{\C}\).
\end{proof}

Let~\(n\) be the rank of the maximal torus in~\(G\) and let~\(W\) be
the Weyl group of~\(G\).  Then \(R\cong
\Z[x_1,\dotsc,x_n,x_1^{-1},\dotsc,x_n^{-1}]^W\),
where the action of~\(W\) comes from the canonical action on~\(T\).
Even more, we have
\[
R\cong \Z[x_1,\dotsc,x_n,x_1^{-1},\dotsc,x_l^{-1}]
\]
for some~\(l\) with \(0\le l\le n\); see for
instance~\cite{Steinberg:Theorem_of_Pittie}.  This ring has
cohomological dimension \(n+1\) because~\(\Z\) has cohomological
dimension~\(1\) and each independent variable adds~\(1\) to the
length of resolutions.

The cohomological dimension of~\(R\) is~\(2\) if and only if
\(n=1\).  The two groups~\(G\) with \(n=1\) are the circle
group~\(\T\) and \(\textup{SU}(2)\).  (The group~\(\textup{SO}(3)\)
has torsion in~\(\pi_1\) and therefore is not covered by
Proposition~\ref{pro:bootstrap_in_KKG}.)  If \(n=1\), then
Theorem~\ref{the:classify_liftings_two-dim} applies to all objects
of~\(\Abel\).  That is, any \(M\inOb\Abel\) has a lifting, and
equivalence classes of liftings are in bijection with
\(\Ext^2_\Abel(M,M[-1])\); this is the even part of
\(\Ext^2_R(M,M[-1])\) with its usual \(\Z/2\)\nb-grading, so that we
may also denote it by \(\Ext^2_R(M,M)^-\).  The category~\(\Abel\)
is even, so that the results of Section~\ref{sec:even_case} apply as
well.  That is, there is a canonical lifting of any \(M\inOb\Abel\),
namely, the direct sum \(\hat M_+\oplus \hat M_-\), where \(\hat
M_+\) and~\(\hat M_-\) are the unique lifting of the even and odd
part of~\(M\), respectively.  Every object \(A\inOb\gen{\C}\) has an
invariant \((M,\delta)\inOb\Abel\delta\) with \(M\defeq \K^G_*(A)\)
and \(\delta\in\Ext^2_R(M,M)^-\);
Theorem~\ref{thm:Ext_enriched_classification} says that
\(A_1,A_2\inOb\gen{\C}\) corresponding to \((M_1,\delta_1)\) and
\((M_2,\delta_2)\) in~\(\Abel\delta\) are isomorphic if and only if
there is a grading preserving \(R\)\nb-module isomorphism \(f\colon
M_1\to M_2\) with \(f\delta_1=\delta_2 f\) in
\(\Ext^2_R(M_1,M_2)^-\).

If \(n=2\), then Theorem~\ref{the:classify_liftings_two-dim} still
applies, among others, to objects of~\(\Abel\) that are free as
Abelian groups.  Groups~\(G\) for which this happens are \(\T^2\),
\(\T\times\textup{SU}(2)\), \(\textup{SU}(2)\times\textup{SU}(2)\),
\(\textup{SU}(3)\), \(\textup{Spin}(5)\), and the simply connected
compact Lie group with Dynkin diagram of type~\(G_2\).  For even
higher rank, we know no useful sufficient criterion for an
\(R\)\nb-module to have a projective resolution of length~\(2\).

We now consider some natural examples of circle actions on
\(\Cst\)\nb-algebras.  Thus \(G=\T\) and \(R=\Z[x,x^{-1}]\) from now
on.

\begin{example}
  \label{exa:invariant_On}
  Consider the Cuntz algebra~\(\mathcal{O}_n\) with its usual gauge
  action, defined by multiplying each generator by \(z\in\T\).
  Then~\(\mathcal{O}_n\rtimes\T\) is Morita--Rieffel equivalent to the
  fixed-point algebra~\(\mathcal{O}_n^\T\).  This is the UHF-algebra
  of type~\(n^\infty\).  It belongs to the bootstrap class, so that
  \(\mathcal{O}_n\in\gen{\C}\), and it has
  \(\K\)\nb-theory~\(\Z[1/n]\).  The generator of the representation
  ring~\(x\) acts on this by multiplication by~\(n\).  Thus
  \[
  M\defeq \K^\T_*(\mathcal{O}_n) \cong \Z[x,x^{-1}]/(x-n),
  \]
  where~\((x-n)\) means the principal ideal generated by~\(x-n\).
  This is concentrated in degree~\(0\) and has a length-1-projective
  resolution
  \begin{equation}
    \label{eq:length-one-resolution_On}
    0 \to \Z[x,x^{-1}] \xrightarrow{x-n} \Z[x,x^{-1}]
    \prto \K^\T_*(\mathcal{O}_n).
  \end{equation}
  Either of these two facts shows that
  \(\Ext^2_{\Z[x,x^{-1}]}(M,M)^-=0\).  Hence~\(\mathcal{O}_n\) is the
  unique object of~\(\gen{\C}\) with \(\K^\T_*(A) \cong
  \Z[x,x^{-1}]/(x-n)\).
\end{example}

\subsection{Cuntz--Krieger algebras}
\label{sec:invariant_Cuntz-Krieger}

Now consider the Cuntz--Krieger algebra~\(\mathcal{O}_A\) with its
usual gauge action; it is defined by an \(n\times n\)-matrix~\(A\)
with entries in \(\{0,1\}\) or more generally in the non-negative
integers, such that no row or column vanishes identically.  The crossed
product~\(\mathcal{O}_A\rtimes\T\) is Morita--Rieffel equivalent to
the fixed-point algebra~\(\mathcal{O}_A^\T\) by
\cite{Pask-Raeburn:K-Theory_of_CKA}*{Theorem 3.2.2 and Lemma 4.1.1}.
The fixed-point algebra is an AF-algebra, and its \(\K_0\)\nb-group is
isomorphic to the direct limit of the iteration sequence
\[
\Z^n \xrightarrow{A^\textup{t}}
\Z^n \xrightarrow{A^\textup{t}}
\Z^n \xrightarrow{A^\textup{t}}
\Z^n \xrightarrow{A^\textup{t}} \dotsb;
\]
the action of the generator~\(x\) is induced by multiplication
with~\(A^\textup{t}\) (see
\cite{Cuntz:topological_Markov_chains_II}*{Proof of Proposition
  3.1}).  In particular, given two Cuntz--Krieger algebras
\(\mathcal{O}_A\) and~\(\mathcal{O}_B\), for degree reasons we have
\(\Ext^2_{\Z[x,x^{-1}]}\bigl(\K_*^\T(\mathcal{O}_A),
\K_{*}^\T(\mathcal{O}_B)\bigr)^-=0\).  Alternatively, we may write
down a projective resolution of \(\K_*^\T(\mathcal{O}_A)\) of
length-\(1\) as in~\eqref{eq:length-one-resolution_On}.  Hence every
grading-preserving \(\Z[x,x^{-1}]\)-module isomorphism
\(\K_*^\T(\mathcal{O}_A)\to\K_*^\T(\mathcal{O}_B)\) lifts to a
\(\KK^\T\)-equivalence.  We get the following characterisation of
\(\KK^\T\)-equivalence for Cuntz--Krieger algebras:

\begin{theorem}
  Let~\(A\) and~\(B\) be finite square matrices with non-negative
  integral entries such that no row or column vanishes identically.
  The following are equivalent:
  \begin{itemize}
  \item The gauge actions on~\(\mathcal{O}_A\) and~\(\mathcal{O}_B\)
    are \(\KK^\T\)\nb-equiv\-a\-lent.
  \item The \(\Z[x,x^{-1}]\)-modules \(\K_*^\T(\mathcal{O}_A)\) and
    \(\K_*^\T(\mathcal{O}_B)\) are isomorphic.
  \item The matrices~\(A\) and~\(B\) are shift equivalent over the
    integers.
  \end{itemize}
\end{theorem}

\begin{proof}
  The equivalence of the first two statements follows from the
  argument above.  For the equivalence of the second and third
  statement, see
  \cite{Lind-Marcus:Intro_to_symbolic_dynamics}*{Theorem~7.5.7}.
\end{proof}

\cite{Boyle:Shift_equivalence_Jordan_from}*{Example~2.13} gives two
irreducible non-negative \(4\times 4\)\nb-matrices~\(A\) and~\(B\)
that are shift equivalent over the integers but not over the
\emph{non-negative} integers.  Then the gauge actions on the purely
infinite simple Cuntz--Krieger algebras~\(\mathcal{O}_A\)
and~\(\mathcal{O}_B\) are \(\KK^\T\)\nb-\hspace{0pt}equivalent by
the previous theorem; but the \emph{ordered}
\(\Z[x,x^{-1}]\)-modules \(\K_0(\mathcal{O}_A^\T)\) and
\(\K_0(\mathcal{O}_B^\T)\) are not isomorphic by
\cite{Lind-Marcus:Intro_to_symbolic_dynamics}*{Theorem~7.5.8}.
Hence the shift automorphisms on the gauge fixed-point
algebras~\(\mathcal{O}_A^\T\) and~\(\mathcal{O}_B^\T\) cannot be
stably conjugate.  By Takai duality, the gauge actions
on~\(\mathcal{O}_A\) and~\(\mathcal{O}_B\) cannot be stably
conjugate.  We cannot expect a Kirchberg--Phillips type
classification result for circle actions unless the fixed-point
algebra is also purely infinite and simple.  For the most useful
gauge actions, the fixed-point algebra is AF, however, so that we
cannot expect isomorphisms in~\(\KK^\T\) to lift to
\Star{}isomorphisms.

\begin{remark}
 It was already observed in
 \cite{Rosenberg-Schochet:Kunneth}*{Proposition 10.4} that the
 \(\KK^\T\)\nb-equiv\-a\-lence class of an object in the
 \(\T\)\nb-equivariant bootstrap class with equivariant
 \(\K\)\nb-theory concentrated in one degree is determined by its
 equivariant \(\K\)\nb-theory.
 \cite{Rosenberg-Schochet:Kunneth}*{\S10} contains some more results
 establishing \(\KK^G\)\nb-equiv\-a\-lence in special cases for
 Hodgkin--Lie groups~\(G\).
\end{remark}

\subsection{On the computation of \texorpdfstring{$\Ext^2$}{Ext²}}
\label{sec:compute_Ext2}

We now describe \(\Ext^*_R(V,W)\) for two general \(R\)\nb-modules
\(V\) and~\(W\), where \(R= \Z[x,x^{-1}]\).  We view an
\(R\)\nb-module~\(V\) as an Abelian group with an
automorphism~\(x_V\), namely, the action of the generator \(x\in
R\).

The ring~\(R\) has a very short \(R\)\nb-bimodule resolution
\[
0 \to R\otimes R \xrightarrow{x\otimes1-1\otimes x}
R\otimes R \xrightarrow{\textup{mult}} R \to 0.
\]
This remains exact when we apply the functor \(\blank\otimes_R V\)
for a left \(R\)\nb-module~\(V\), and this gives a short exact
sequence of \(R\)\nb-modules
\[
0\to
R\otimes V \xrightarrow{x\otimes1-1\otimes x_V} R\otimes V
\xrightarrow{\textup{mult}} V \to 0
\]
for any \(R\)\nb-module~\(V\).  Given another \(R\)\nb-module~\(W\),
the long exact cohomology sequence for this short exact sequence becomes
\begin{multline*}
  0 \to \Hom_R(V,W) \to \Hom_R(R\otimes V,W) \to
  \Hom_R(R\otimes V,W)
  \\\to \Ext^1_R(V,W) \to \Ext^1_R(R\otimes V,W) \to
  \Ext^1_R(R\otimes V,W)
  \\\to \Ext^2_R(V,W) \to \Ext^2_R(R\otimes V,W) \to
  \Ext^2_R(R\otimes V,W) \to 0.
\end{multline*}
This simplifies considerably because
\[
\Ext^n_R(R\otimes V,W) \cong \Ext^n_\Z(V,W)
\]
by adjoint associativity.  Thus we get a long exact sequence
\begin{multline*}
  0 \to \Hom_R(V,W) \to \Hom_\Z(V,W) \to \Hom_\Z(V,W)
  \\\to \Ext^1_R(V,W) \to \Ext^1_\Z(V,W) \to \Ext^1_\Z(V,W) \to
  \Ext^2_R(V,W) \to 0.
\end{multline*}
Here the maps \(\Hom_\Z(V,W) \to \Hom_\Z(V,W)\) and
\(\Ext^1_\Z(V,W) \to \Ext^1_\Z(V,W)\) are \(f\mapsto (x_W)_*f-
(x_V)^*f\), using the automorphisms \(x_V\) and~\(x_W\) of
\(V\) and~\(W\).  Hence
\begin{equation}
  \label{eq:Ext_2_Laurent}
  \Ext^2_R(V,W) \cong \coker \bigl((x_W)_*-(x_V)^*\bigr)\colon
  \Ext^1_\Z(V,W) \to \Ext^1_\Z(V,W).
\end{equation}

\begin{remark}
  We may view this cokernel as the first Hochschild cohomology
  for~\(R\) with coefficients in \(\Ext^1_\Z(V,W)\) with the
  induced \(R\)\nb-bimodule structure.  The kernel of this map
  is the zeroth Hochschild cohomology.  The above long exact
  sequence is equivalent to a spectral sequence
  \[
  \textup{HH}^p\bigl(R,\Ext^q_\Z(V,W)\bigr) \Rightarrow \Ext^{p+q}_R(V,W).
  \]
\end{remark}

Eusebio Gardella shows in \cites{Gardella:Classif_circle_actions_I,
Gardella:Classif_circle_actions_II} that for circle actions on unital
Kirchberg algebras~\(A\) with the Rokhlin property, such that~\(A\)
satisfies the Universal Coefficient Theorem and has finitely generated
\(\K\)\nb-theory groups, the action of the generator of~\(R\) on
equivariant \(\K\)\nb-theory is the identity (this is analogous to the
situation of finite group actions with the Rokhlin
property, see~\cite{Phillips:Freeness_actions_finite_groups}), and
equivariant \(\K\)\nb-theory together with the unit class is a complete
invariant.  Moreover, \(\K_0(A)\cong \K_1(A) \cong \K_0(A^\T)\oplus
\K_1(A^\T)\), and every pair \((G_0,G_1)\) of finitely generated
Abelian groups with any unit class in~\(G_0\) may be realised as
\(\bigl(\K_0(A^\T),\K_1(A^\T)\bigr)\).  

If~\(x\) acts identically on~\(V\) and~\(W\), then \(\Ext^2_R(V,W)
\cong \Ext^1_\Z(V,W)\) by~\eqref{eq:Ext_2_Laurent}.  Therefore,
there must be a unique obstruction class in
\(\Ext^1_\Z\bigl(\K_*(A^\T), \K_{*+1}(A^\T)\bigr)\) that comes from
a Rokhlin action on a unital Kirchberg algebra.  We do
not know, however, which obstruction class this is.

Another classification result for Rokhlin actions of finite groups
on Kirchberg algebras was proved by Masaki
Izumi~\cite{Izumi:Finite_group}.

\subsection{Nekrashevych's \texorpdfstring{$\Cst$}{C*}-algebras of self-similar groups}
\label{sec:Nekrashevych}

Nekrashevych~\cite{Nekrashevych:Cstar_selfsimilar} constructs purely
infinite simple \(\Cst\)\nb-algebras with a gauge action of~\(\T\)
from self-similar groups.  He proves that the conjugacy class of this
gauge action essentially determines the underlying self-similar group
and hence is a very fine invariant.  This is, however, far from true
for the \(\KK^\T\)-equivalence class.

We consider only the particular case considered in
\cite{Nekrashevych:Cstar_selfsimilar}*{Theorem 4.8} to use
Nekrashevych's \(\K\)\nb-theory computation.  The self-similar group~\(G\)
in question is the iterated monodromy group of a post-critically
finite, hyperbolic, rational function~\(f\) on~\(\hat{\C}\).  Let~\(n\)
be the (mapping) degree of this rational function, that is, each
non-critical point has precisely~\(n\) preimages.  The function~\(f\)
has at most finitely many attracting cycles; let their lengths be
\(\ell_1,\ldots,\ell_c\), listed with repetitions.  Thus~\(f\) has
\(c\)~attracting cycles.

It is asserted in
\cite{Nekrashevych:Cstar_selfsimilar}*{Theorem 4.8} that the
\(\K\)\nb-theory of the gauge fixed-point algebra of the
\(\Cst\)\nb-algebra~\(\mathcal{O}_G\) associated to~\(f\) is
\(\Z[1/n]\)~in even degrees and \(\Z^{k-1}\)~in odd degrees,
where \(k=\sum_{i=1}^c \ell_i\).  We can be more precise: the
proof of \cite{Nekrashevych:Cstar_selfsimilar}*{Theorem 4.8}
also gives the \(\T\)\nb-equivariant \(\K\)\nb-theory
of~\(\mathcal{O}_G\).

First, the fixed-point algebra is Morita--Rieffel equivalent to the
crossed product in this case, so that the \(\K\)\nb-theory of the
gauge fixed-point algebra is isomorphic to the \(\T\)\nb-equivariant
\(\K\)\nb-theory and carries a \(\Z[x,x^{-1}]\)-module structure.  The
action of~\(x\) on this module is given by multiplication by~\(n\) on
the even part, as for the Cuntz algebra~\(\mathcal{O}_n\).  Thus
\(\K_0^\T(\mathcal{O}_G)\cong R/(x-n)\) has a projective resolution of
length~\(1\).

The odd part is the quotient of
\(H=\Z^{\ell_1}\oplus\dotsb\oplus\Z^{\ell_c}\) by the
diagonally embedded copy of~\(\Z\).  We may view~\(H\) as the
space of functions from the union of the attracting cycles
of~\(f\) to~\(\Z\).  The generator~\(x\) acts like~\(f\) on
these functions, that is, it is a cyclic permutation in each
copy of~\(\Z^{\ell_i}\).  Thus we get the quotient of the
module
\[
V \defeq \bigoplus_{i=1}^c \Z[x_i,x_i^{-1}]/(x_i^{\ell_i}-1)
\]
by the copy of~\(\Z\) generated by \(N_i\defeq
1+x_i+\dotsb+x_i^{\ell_i-1}\) in each component.

\begin{lemma}
  \label{lem:Nekrashevych_no_Ext}
  \(\Ext^2_R\bigl(\K^\T_*(\mathcal{O}_G),\K^\T_{*+1}(
  \mathcal{O}_G)\bigr)=0\).
\end{lemma}

\begin{proof}
  Let \(\K_*\defeq \K^\T_*(\mathcal{O}_G)\).  Since \(\K_1\) is free
  as an Abelian group, \(\Ext^1_\Z(\K_1,\K_0)=0\) and hence
  \(\Ext^2_R(\K_1,\K_0)=0\).  Since~\(\K_0\) has a projective
  \(R\)\nb-module resolution of length~\(1\)
  by~\eqref{eq:length-one-resolution_On}, \(\Ext^2_R(\K_1,\K_0)=0\) as
  well.
\end{proof}

Since~\(\mathcal{O}_G^\T\) is the \(\Cst\)\nb-algebra of an amenable
groupoid, it belongs to the bootstrap class.  Hence so does the
Morita--Rieffel equivalent \(\Cst\)\nb-algebra
\(\mathcal{O}_G\rtimes\T\), so our classification results apply by
Proposition~\ref{pro:bootstrap_in_KKG}.
Lemma~\ref{lem:Nekrashevych_no_Ext} shows that, up to
circle-equivariant \(\KK\)\nb-equivalence, the
\(\Cst\)\nb-algebra~\(\mathcal{O}_G\) is classified completely by
the \(\Z/2\)\nb-graded \(\Z[x,x^{-1}]\)\nb-module
\(\K^\T_*(\mathcal{O}_G)\).  This module only remembers the
degree~\(d\) of~\(f\) and the multiset of lengths~\(\ell_i\), so it
is a rather coarse invariant.

It would be very interesting to refine our invariant to detect the
conjugacy class of the gauge action because this determines the
action of~\(f\) on its Julia set up to topological conjugacy by
Nekrashevych's main result
(\cite{Nekrashevych:Cstar_selfsimilar}*{Section 4.2}).
Unfortunately, we know no useful refinements for our invariant.
Both for~\(\mathcal{O}_G\) and for Cuntz--Krieger algebras, the
fixed-point algebra of the gauge action has a unique trace.  For
Cuntz--Krieger algebras, the order structure on~\(\K_0^\T\) gives a
finer invariant (see Section~\ref{sec:invariant_Cuntz-Krieger}), but
for~\(\mathcal{O}_G\), the group \(\K_0^\T(\mathcal{O}_G) \cong
\Z[1/d]\) carries no interesting order structure.

\section{Kasparov theory for \texorpdfstring{$\Cst$}{C*}-algebras over unique path spaces}
\label{sec:KK_X}

Let~\(X\) be a finite \(T_0\)\nb-space.  In this section, we
consider Kirchberg's ideal-related \(\KK\)-theory
\(\Tri\defeq\KK(X)\), following~\cites{Meyer-Nest:Bootstrap,
  Meyer-Nest:Filtrated_K}.  Let \(i_x\C\inOb\Tri\) denote the
\(\Cst\)\nb-algebra of complex numbers~\(\C\) equipped with the
continuous map \(\Prim(\C)\to X\) taking the unique element of
\(\Prim(\C)\) to \(x\in X\).  The bootstrap class \(\Boot(X)\)
in~\(\Tri\) is the localising subcategory generated by the
collection \(\{i_x\C\mid x\in X\}\) of one-dimensional
\(\Cst\)\nb-algebras over~\(X\) (see
\cite{Meyer-Nest:Bootstrap}*{Definition 4.11}).

We apply the homological machinery
from~\cite{Meyer-Nest:Homology_in_KK} to the family of functors
represented by the objects~\(i_x\C\), respectively.  Let~\(A(U_x)\)
be the distinguished ideal of~\(A\) corresponding to the minimal
open neighbourhood~\(U_x\) of~\(x\) in~\(X\).  The adjointness
relations in \cite{Meyer-Nest:Bootstrap}*{Proposition 3.13}
specialise to
\begin{equation}
  \label{eq:adjointness_relation}
 \KK_*(X;i_x\C,A)\cong\KK_*\bigl(\C,A(U_x)\bigr)
 \cong \K_*\bigl(A(U_x)\bigr).
\end{equation}
For \(x\in X\), consider the stable homological functor
\[
F_x\colon\Tri\to\Ab_\mathrm{c}^{\Z/2},\quad A\mapsto\KK_*(X;i_x\C,A)
\]
and the homological ideal \(\Ideal_x\defeq\ker F_x\).
Since
\[
\KK_*(X;i_x\C,A)\cong\Hom_{\Ab_\mathrm{c}^{\Z/2}}\bigl(\Z[0],F_x(A)\bigr),
\]
the adjoint functor~\(F_x^\lad\) takes the free rank-one Abelian group
in even degree to the object~\(i_x\C\).
\cite{Meyer-Nest:Homology_in_KK}*{Theorem 57} implies that~\(F_x\) is
the universal \(\Ideal_x\)\nb-exact functor and that~\(\Ideal_x\) has
enough projective objects.

Now we consider the homological ideal
\(\Ideal\defeq\bigcap_{x\in X}\Ideal_x\).
\cite{Meyer-Nest:Bootstrap}*{Theorem 4.17} gives \(\KK_*(X;A,B)=0\)
for all \(\Ideal\)\nb-contractible~\(B\)
if and only if~\(A\) belongs to~\(\Boot(X)\).
By \cite{Meyer-Nest:Homology_in_KK}*{Proposition~55},
the ideal~\(\Ideal\) has enough projective objects.  An argument as in
\cite{Meyer-Nest:Filtrated_K}*{Section  4.3} shows that the universal
\(\Ideal\)\nb-exact stable homological functor is
\[
\XK\defeq\KK_*(X;\Repr,\blank)\colon\Tri
\to \Modc{\KK_*(X;\Repr,\Repr)^\op},
\]
where \(\Repr\defeq\bigoplus_{x\in X} i_x\C\) and
\(\Modc{\KK_*(X;\Repr,\Repr)^\op}\) denotes the category of countable
\(\Z/2\)-graded \emph{right} modules over the \(\Z/2\)-graded ring
\(\KK_*(X;\Repr,\Repr)\).

Partially order~\(X\) by \(x\preceq y\) if and only if
\(x\in\overline{\{y\}}\), if and only if \(y\in U_x\).
Equation~\eqref{eq:adjointness_relation} implies
\[
\KK_*(X;i_x\C,i_y\C)\cong
\begin{cases}
 \Z[0]& \textup{if \(x\preceq y\),}\\
 0 & \textup{otherwise,}
\end{cases}
\]
for all \(x,y\in X\).  The proof of~\eqref{eq:adjointness_relation}
shows that the generator of \(\KK_0(X;i_x\C,i_y\C)=\Z\) for
\(x\preceq y\) is the class~\(i_x^y\) of the identity map on~\(\C\),
viewed as a \Star{}homomorphism over~\(X\) from~\(i_x\C\)
to~\(i_y\C\).  Since \(i_y^z\circ i_x^y=i_x^z\), the \(\Z/2\)-graded
ring \(\KK_*(X;\Repr,\Repr)^\op\) is isomorphic to the integral
incidence algebra~\(\Z[X]\) of the poset~\((X,\preceq)\) in even
degree and vanishes in odd degree; here we use the convention
that~\(\Z[X]\) is the free Abelian group generated by
elements~\(f_{x\preceq y}\) for all pairs~\((x,y)\) with \(x\preceq
y\); the multiplication is defined by \(f_{x\preceq y} f_{y\preceq
  z}= f_{x\preceq z}\).

We write \(x\to y\) for \(x,y\in X\) if \(x\succ y\) and there is no
\(z\in X\) with \(x\succ z\succ y\).  Then \(x\succ y\) if and only if
there is a chain \(x=x_0\to x_1\to \dotsb\to x_\ell= y\) with some
\(x_1,\dotsc,x_{\ell-1}\in X\).

\begin{definition}
  \label{def:unique_path_space}
  A finite \(T_0\)\nb-space~\(X\) is called a \emph{unique path space}
  if the chain \(x=x_0\to x_1\to \dotsb\to x_\ell= y\) is unique for
  all \(x,y\in X\) with \(x\succ y\).
\end{definition}

If~\(X\) is a unique path space, then~\(\Z[X]\) is the integral path
algebra of the quiver \((X,\to)\), where paths are concatenated such
that arrows point to the left.

There is a canonical family of orthogonal idempotent elements \(e_x\in
\Z[X]\) for \(x\in X\) with \(\sum_{x\in X}e_x=1\).  Viewing
\(\Z[X]\)-bimodules as modules over the ring \(\Z[X]\otimes_\Z\Z[X^\op]\),
we see that \(P_{x\otimes y}\defeq\Z[X]e_x\otimes e_y\Z[X]\) is a
projective \(\Z[X]\)-bimodule (corresponding to the idempotent
\(e_x\otimes e_y\)).

\begin{lemma}
  \label{lem:resolution_ups}
  Let~\(X\) be a unique path space.  There is a length-one projective
  bimodule resolution
  \[
  0\to \bigoplus_{y\to x} P_{x\otimes y}
  \to \bigoplus_{x\in X} P_{x\otimes x}
  \to \Z[X]\to 0.
  \]
  The second map is determined by the maps \(P_{x\otimes x}\to
  \Z[X]\), \(a\otimes b\mapsto a\cdot b\), the first one by the maps
  \(P_{x\otimes y} \to P_{x\otimes x}\oplus P_{y\otimes y}\),
  \(a\otimes b\mapsto (a\otimes f_{x\preceq y}b,
  - a f_{x\preceq y}\otimes b)\).
\end{lemma}

\begin{proof}
  We clearly have a complex of bimodule maps.  The underlying
  Abelian groups of the three modules are free.  To verify
  exactness, we choose the following canonical bases.  A basis
  for~\(\Z[X]\) is given by paths \(p=(x_1\to\cdots\to x_k)\) of
  non-negative length in~\(X\); similarly, a basis for the middle
  group is given by paths \(p_l=(x_1\to\cdots\to x_l^*\to\cdots\to
  x_k)\) with a marked vertex position \(l\in\{1,\ldots,k\}\), and a
  basis for the left group is given by paths
  \(p_{l,l+1}=(x_1\to\cdots\to x_l^*\to x_{l+1}^*\to \cdots\to
  x_k)\) with two marked, consecutive vertices \(x_l\)
  and~\(x_{l+1}\).  In this picture, the right map simply forgets the
  position of the marked vertex.  Hence it is surjective.  The left
  map takes a doubly marked path \(p_{l,l+1}\) to the linear
  combination \(p_l-p_{l+1}\) of singly marked paths.  Since this map
  does not change the underlying path~\(p\), it suffices to check
  injectivity of the left map on elements of the form
  \(\sum_{l=1}^{k-1} n_l p_{l,l+1}\).  Applying the left map yields
  \(n_1 p_1+\left( \sum_{l=2}^{k-1} (n_l-n_{l-1}) p_{l}\right) -
  n_{k-1} p_k\).  But this sum can only vanish if \(n_1=0\), hence
  \(n_2=0\), and so on.  Finally, we show exactness in the middle.
  The kernel of the right map is generated by elements of the form
  \(\sum_{l=1}^{k} n_l p_l\) with \(\sum_{l=1}^{k} n_l
  =0\).  Rewriting such an element as \(\sum_{l=1}^{k-1}
  \left(\sum_{j=1}^l n_j\right) (p_l - p_{l+1})\) using
  \(-\sum_{l=1}^{k-1}n_l=n_k\) shows that it belongs to the image of
  the left map.
\end{proof}

\begin{proposition}
   \label{pro:UP_dimension}
   If~\(X\) is a unique path space, then \(\KK_*(X;\Repr,\Repr)^\op\)
   has cohomological dimension at most~\(2\).
\end{proposition}

In fact, it is easy to see that the cohomological dimension of
\(\KK_*(X;\Repr,\Repr)^\op\) is equal to~\(2\) unless the space~\(X\)
is discrete (in which case it is~\(1\)).

\begin{proof}
  Tensoring the above short exact bimodule sequence over~\(\Z[X]\)
  with a left \(\Z[X]\)-module~\(V\) gives the short exact sequence
  \[
  0 \to \bigoplus_{y\to x} P_{x\otimes y}\otimes V
  \to \bigoplus_{x\in X} P_{x\otimes x}\otimes V
  \to V \to 0.
  \]
  We have \(P_{x\otimes y}\otimes_{\Z[X]} V\cong \Z[X]e_x\otimes_\Z
  V_y\), where~\(V_y\) is the entry group of the module~\(V\) at~\(y\).
  It follows that
  \(\Ext^n_{\Z[X]}(P_{x\otimes y}\otimes V,W) \cong \Ext^n_\Z(V_y,W_x)\).
  The long exact cohomology sequence for the functor
  \(\Hom_{\Z[X]}(\blank,W)\) applied to the above short exact sequence
  is thus of the form
  \begin{multline*}
    0 \to \Hom_{\Z[X]}(V,W) \to \bigoplus_{x\in X}\Hom_\Z(V_x,W_x)
    \to \bigoplus_{y\to x}\Hom_\Z(V_y,W_x)
    \\\to \Ext^1_{\Z[X]}(V,W) \to \bigoplus_{x\in X}\Ext^1_\Z(V_x,W_x)
    \to \bigoplus_{y\to x}\Ext^1_\Z(V_y,W_x)
    \\\to \Ext_{\Z[X]}^2(V,W) \to 0,
  \end{multline*}
  and \(\Ext_{\Z[X]}^n\) vanishes for \(n\geq 3\) because
  \(\Ext_{\Z}^n\) vanishes for \(n\geq 2\).
\end{proof}

The maps \(\bigoplus_{x\in X}\Ext^n_\Z(V_x,W_x) \to \bigoplus_{y\to
  x}\Ext^n_\Z(V_y,W_x)\) in the exact sequence above are the sum of
the maps
\[
\Ext^n_\Z(V_x,W_x)\oplus\Ext^n_\Z(V_y,W_y)
\xrightarrow{\bigl((i_W)_*,-(i_V)^*\bigr)}\Ext^n_\Z(V_y,W_x),
\]
induced by the arrows \(i\colon y\to x\) in~\(X\).  This gives a
scheme for computing the groups \(\Ext^n_{\Z[X]}(V,W)\).  As in
Section~\ref{sec:compute_Ext2}, the above long exact sequence is
equivalent to a spectral sequence
\[
\textup{HH}^p\bigl(\Z[X],\Ext^q_\Z(V,W)\bigr)
\Rightarrow \Ext^{p+q}_{\Z[X]}(V,W).
\]

\begin{definition}
  A \(\Cst\)\nb-algebra over~\(X\) is called a \emph{Kirchberg
    \(X\)\nb-algebra} if it is separable, tight (see
    \cite{Meyer-Nest:Bootstrap}*{Definition 5.1}),
    \(\mathcal{O}_\infty\)\nb-absorbing and nuclear.
\end{definition}

Combining Theorem~\ref{thm:Ext_enriched_classification} and
Proposition~\ref{pro:UP_dimension} with Kirchberg's Classification
Theorem in~\cite{Kirchberg:Michael}, we get the following purely
algebraic complete classification of Kirchberg \(X\)\nb-algebras in
the bootstrap class~\(\Boot(X)\):

\begin{corollary}
  Let~\(X\) be a unique path space.  Then the functor~\(\XK\delta\)
  induces a bijection between the set of \Star{}isomorphism classes
  over~\(X\) of stable Kirchberg \(X\)\nb-algebras in the bootstrap
  class~\(\Boot(X)\) and the set of isomorphism classes in the
  category \(\Modc{\Z [X]}^{\Z/2} \delta\).  Every isomorphism in
  \(\Modc{\Z[X]}^{\Z/2}\delta\) lifts to a \Star{}iso\-mor\-phism
  over~\(X\).
\end{corollary}

\begin{example}
  If \(X=\bullet\) is the one-point space, then~\(\Z[X]\) is
  simply the ring of integers, which has global dimension~\(1\).
  Hence, in this case, the functor~\(\XK\delta\) reduces to plain
  \(\Z/2\)\nb-graded \(\K\)\nb-theory.
\end{example}

\begin{example}
  \label{exa:compare_class_2}
  Let \(X=\bullet\to\bullet\) be the two-point Sierpi\'nski space.
  Stable Kirchberg \(X\)\nb-algebras in~\(\Boot(X)\) are essentially
  the same as stable extensions of UCT Kirchberg algebras.
  R\o rdam~\cite{Rordam:Classification_extensions} classified these by
  their six-term exact sequences in \(\K\)\nb-theory.  Moreover, every
  six-term exact sequence of countable Abelian groups arises as the
  \(\K\)\nb-theory sequence of a stable Kirchberg \(X\)\nb-algebra
  in~\(\Boot(X)\).  The equivalence between R\o rdam's invariant and
  ours becomes obvious by the following direct computation: given two
  objects \(G_1\xrightarrow{\varphi} G_2\) and \(H_1\xrightarrow{\psi}
  H_2\) in~\(\Mod{\Z[X]}\), we have natural isomorphisms
  \begin{multline*}
    \Ext^2_{\Z[X]}(G_1\xrightarrow{\varphi} G_2,H_1\xrightarrow{\psi} H_2)
    \cong \Ext^2_{\Z[X]}\bigl(\ker(\varphi)\to 0,0\to\coker(\psi)\bigr)
    \\\cong \Ext^1_{\Z[X]}\bigl(0\to \ker(\varphi),0\to\coker(\psi)\bigr)
    \cong \Ext^1_{\Z}\bigl(\ker(\varphi),\coker(\psi)\bigr).
  \end{multline*}
  The group \(\Ext^1_{\Z}\bigl(\ker(\varphi),\coker(\psi)\bigr)\) is
  in natural bijection to the set of equivalence classes of exact
  sequences of the form
  \[
  H_1\xrightarrow{\psi} H_2\to E\to G_1\xrightarrow{\varphi} G_2.
  \]
  In fact, our invariant factors through R\o rdam's; it remembers
  isomorphism classes but forgets certain morphisms.
\end{example}

\begin{example}
  Let~\(X\) be totally ordered (for two points, this is
  Example~\ref{exa:compare_class_2}).  Then filtrated \(\K\)\nb-theory
  is a complete invariant for objects in~\(\Boot(X)\) by the main
  result of~\cite{Meyer-Nest:Filtrated_K}.  Since totally ordered
  spaces are unique path spaces, we now have two seemingly different
  complete invariants for objects in~\(\Boot(X)\).  Both invariants
  must contain exactly the same information.  The authors, however, do
  not understand the relationship between these two invariants.  If,
  for instance, \(X=\bullet\to\bullet\to\bullet\), then the issue is to
  relate elements in \(\Ext^2_{\Z[X]}(G_1\to G_2\to G_3,H_1\to H_2\to
  H_3)\) to diagrams of the form
  \[
  \begin{tikzpicture}[scale=0.75,baseline=(current bounding box.west)]
    \node (0t) at (3,4) {\(0\)};
    \node (01) at (1,0) {\(0\)};
    \node (02) at (3,0) {\(0\)};
    \node (03) at (5,0) {\(0\)};
    \node (H1) at (0,1) {\(H_1\)};
    \node (H2) at (1,2) {\(H_2\)};
    \node (H3) at (2,3) {\(H_3\)};
    \node (G1) at (4,3) {\(G_1\)};
    \node (G2) at (5,2) {\(G_2\)};
    \node (G3) at (6,1) {\(G_3\)};
    \node (K)  at (2,1) {\(K\)};
    \node (L)  at (3,2) {\(L\)};
    \node (M)  at (4,1) {\(M\)};
    \begin{scope}[cdar]
      \draw (H1) -- (H2);
      \draw (H2) -- (H3);
      \draw (H3) -- (0t);
      \draw (H1) -- (01);
      \draw (H2) -- (K);
      \draw (H3) -- (L);
      \draw (0t) -- (G1);
      \draw (01) -- (K);
      \draw (K)  -- (L);
      \draw (L)  -- (G1);
      \draw (K)  -- (02);
      \draw (L)  -- (M);
      \draw (G1) -- (G2);
      \draw (02) -- (M);
      \draw (M)  -- (G2);
      \draw (M)  -- (03);
      \draw (G2) -- (G3);
      \draw (03) -- (G3);
    \end{scope}
  \end{tikzpicture}
  \]
  such that all squares commute and certain exactness conditions hold.
\end{example}

\section{Graph \texorpdfstring{$\Cst$}{C*}-algebras}
\label{sec:graph}

If the finite \(T_0\)\nb-space~\(X\) is not a unique path space, then
we may still classify those objects~\(A\) of~\(\Boot(X)\) for which
\(\XK(A)\) has a projective resolution of length~\(2\).  We are going
to show that graph \(\Cst\)\nb-algebras with finitely many ideals have
this property.  Even better, we may compute their obstruction classes
in terms of the Pimsner--Voiculescu type sequence that computes their
\(\K\)\nb-theory.

\subsection{A computation of obstruction classes}
\label{sec:compute_obstruction_classes}

First we prove a general result in the abstract setting of a
triangulated category~\(\Tri\) with a universal \(\Ideal\)\nb-exact
stable homological functor \(F\colon \Tri\to\Abel\); we also impose
the parity assumptions of Section~\ref{sec:even_case}.  For certain
objects in~\(\Tri\) that are constructed from a length\nb-\(2\)
projective resolution in~\(\Abel\), we compute the obstruction class
explicitly.  Let
\begin{equation}
  \label{eq:cone_MQQM}
  0 \to M_1 \xrightarrow{\partial_2} Q_1
  \xrightarrow{\partial_1} Q_0
  \xrightarrow{\varepsilon} M_0
  \to 0
\end{equation}
be an exact chain complex in~\(\Abel_+\) with projective objects
\(M_1\), \(Q_1\) and~\(Q_0\).  The adjoint functor~\(F^\lad\) on
projective objects of~\(\Abel\) gives objects \(\hat{M}_1\),
\(\hat{Q}_1\) and~\(\hat{Q}_0\) of~\(\Tri\) lifting \(M_1\), \(Q_1\)
and~\(Q_0\), and maps \(\hat{\partial}_2\in\Tri(\hat{M}_1,\hat{Q}_1)\) and
\(\hat{\partial}_1\in\Tri(\hat{Q}_1,\hat{Q}_0)\) lifting \(\partial_2\) and~\(\partial_1\).
Embed~\(\hat{\partial}_1\) into an exact triangle
\[
\hat{Q}_1 \xrightarrow{\hat{\partial}_1}
\hat{Q}_0 \xrightarrow{p}
A \xrightarrow{r}
\hat{Q}_1[1].
\]
The long exact sequence for~\(F\) applied to this triangle has the
form
\[
\dotsb \to Q_1
\xrightarrow{\partial_1} Q_0
\xrightarrow{F(p)} F(A)
\xrightarrow{F(r)} Q_1[1]
\xrightarrow{\partial_1[1]} Q_0[1]
\to \dotsb
\]
Since the cokernel \(M_0\inOb\Abel_+\) of~\(\partial_1\) and the kernel
\(M_1[1]\inOb\Abel_-\) of~\(\partial_1[1]\) have different parity, we get
\[
F(A)\cong M_0\oplus M_1[1].
\]
This has the following projective resolution of length~\(2\):
\begin{equation}
  \label{eq:FA_resolution}
  0 \to M_1
  \xrightarrow{\partial_2} Q_1
  \xrightarrow{\partial_1} Q_0\oplus M_1[1]
  \xrightarrow{(\varepsilon, \id_{M_1[1]})} M_0\oplus M_1[1]
  \to 0.
\end{equation}

\begin{theorem}
  \label{the:obstruction_class_cone}
  The obstruction class of~\(A\) is the class of the \(2\)\nb-step
  extension~\eqref{eq:cone_MQQM} in \(\Ext^2_\Abel(M_0,M_1)\), which
  we embed as a direct summand into
  \begin{multline*}
    \Ext^2_\Abel(F(A),F(A)[-1])
    \cong \Ext^2_\Abel(M_0,M_0[-1]) \oplus \Ext^2_\Abel(M_0,M_1)
    \\\oplus \Ext^2_\Abel(M_1[1],M_0[-1]) \oplus \Ext^2_\Abel(M_1[1],M_1).
  \end{multline*}
  Let
  \begin{equation}
    \label{eq:cone_MQQM_prime}
    0 \to M'_1 \xrightarrow{\partial'_2} Q'_1
    \xrightarrow{\partial'_1} Q'_0
    \xrightarrow{\varepsilon'} M'_0
    \to 0
  \end{equation}
  be another exact chain complex in~\(\Abel\) with even projective
  objects \(M'_1\), \(Q'_1\) and~\(Q'_0\), and let~\(A'\) be the cone
  of the lifting~\(\hat{\partial}'_1\) of~\(\partial_1\).  Then \(A\cong A'\) if and
  only if there is a commutative diagram
  \begin{equation}
    \label{eq:cone_MQQM_compare}
    \begin{tikzpicture}[baseline=(current bounding box.west)]
      \matrix(m)[cd,column sep=5em]{
        M_1&Q_1\oplus Q'_0&Q_0\oplus Q'_0&M_0\\
        M'_1&Q_0\oplus Q'_1&Q_0\oplus Q'_0&M'_0\\
      };
      \begin{scope}[cdar]
        \draw[>->] (m-1-1) -- node {\((\partial_2,0)\)} (m-1-2);
        \draw (m-1-2) -- node {\((\partial_1,\id_{Q'_0})\)} (m-1-3);
        \draw[->>] (m-1-3) -- node {\((\varepsilon,0)\)} (m-1-4);
        \draw[>->] (m-2-1) -- node {\((0,\partial'_2)\)} (m-2-2);
        \draw (m-2-2) -- node {\((\id_{Q_0},\partial'_1)\)} (m-2-3);
        \draw[->>] (m-2-3) -- node {\((0,\varepsilon')\)} (m-2-4);
        \draw (m-1-1) -- node {\(\varphi_1\)} node[swap] {\(\cong\)} (m-2-1);
        \draw (m-1-2) -- node {\(\varphi_2\)} node[swap] {\(\cong\)} (m-2-2);
        \draw (m-1-3) -- node {\(\varphi_3\)} node[swap] {\(\cong\)} (m-2-3);
        \draw (m-1-4) -- node {\(\varphi_4\)} node[swap] {\(\cong\)} (m-2-4);
      \end{scope}
    \end{tikzpicture}
  \end{equation}
  in~\(\Abel\), where the maps \(\varphi_i\) for \(i=1,2,3,4\) are
  isomorphisms.
\end{theorem}

\begin{proof}
  We first compute the obstruction class of~\(A\).  For this, we
  compare~\(A\) to the canonical lifting of~\(F(A)\).  The latter is
  the direct sum of the canonical lifting of~\(M_0\)
  with~\(\hat{M}_1[1]\).  To lift~\(M_0\) canonically, we first embed
  \(\hat{\partial}_2\colon \hat{M}_1\to \hat{Q}_1\) in an exact triangle
  \[
  \hat{M}_1
  \xrightarrow{\hat{\partial}_2} \hat{Q}_1
  \xrightarrow{u} D
  \xrightarrow{v} \hat{M}_1[1].
  \]
  Then \(F(D)\cong \coker \partial_2 \cong \ker \varepsilon\).  The UCT gives
  \(\Tri_0(D,\hat{Q}_0) \cong \Abel(\ker \varepsilon,Q_0)\) for parity
  reasons.  Hence there is a unique \(x\in \Tri_0(D,\hat{Q}_0)\) for which
  \(F(x)\) is the inclusion of \(\ker \varepsilon\) into~\(Q_0\).  The cone
  of~\(x\) is the canonical lifting of~\(M_0\).  Since direct sums of exact
  triangles remain exact, the canonical lifting of~\(F(A)\) is the cone of
  the map \((x,0)\colon D \to \hat{Q}_0\oplus \hat{M}_1[1]\).

  The map \(\hat{\partial}_2\circ \hat{\partial}_1=0\) is part of an exact triangle
  \[
  \hat{M}_1
  \xrightarrow{0} \hat{Q}_0
  \xrightarrow{i_1} \hat{Q}_0\oplus \hat{M}_1[1]
  \xrightarrow{p_2} \hat{M}_1[1],
  \]
  where~\(i_1\) is the inclusion of the first summand and~\(p_2\) the
  projection onto the second summand.  The octahedral axiom applied to
  \(\hat{\partial}_1\) and~\(\hat{\partial}_2\) gives maps \(\bar{x}\colon D\to
  \hat{Q}_0\), \(y\colon D\to \hat{M}_1[1]\) and \(t\colon
  \hat{M}_1[1]\to A\) such that the diagram
  \[
  \begin{tikzpicture}
    \matrix(m)[cd,row sep=6ex]{
      \hat{M}_1&\hat{Q}_1&D&\hat{M}_1[1]\\
      \hat{M}_1&\hat{Q}_0&\hat{Q}_0\oplus \hat{M}_1[1]&\hat{M}_1[1]\\
      0&A&A&0\\
      \hat{M}_1[1]&\hat{Q}_1[1]&D[1]&\hat{M}_1[2]\\
    };
    \begin{scope}[cdar]
      \draw (m-1-1) -- node {\(\hat{\partial}_2\)} (m-1-2);
      \draw (m-1-2) -- node {\(u\)} (m-1-3);
      \draw (m-1-3) -- node {\(v\)} (m-1-4);
      \draw (m-2-1) -- node {\(0\)} (m-2-2);
      \draw (m-2-2) -- node {\(i_1\)} (m-2-3);
      \draw (m-2-3) -- node {\(p_2\)} (m-2-4);
      \draw (m-3-1) -- (m-3-2);
      \draw (m-3-3) -- (m-3-4);
      \draw (m-4-1) -- node {\(\hat{\partial}_2[1]\)} (m-4-2);
      \draw (m-4-2) -- node {\(u[1]\)} (m-4-3);
      \draw (m-4-3) -- node {\(v[1]\)} (m-4-4);
      \draw (m-1-2) -- node {\(\hat{\partial}_1\)} (m-2-2);
      \draw (m-2-1) -- (m-3-1);
      \draw (m-2-2) -- node {\(p\)} (m-3-2);
      \draw (m-2-4) -- (m-3-4);
      \draw (m-3-1) -- (m-4-1);
      \draw (m-3-2) -- node {\(r\)} (m-4-2);
      \draw (m-3-3) -- node {\(u[1]\circ r\)} (m-4-3);
      \draw (m-3-4) -- (m-4-4);
    \end{scope}
    \draw[->,dotted] (m-1-3) -- node {\(
      \left(\begin{smallmatrix} \bar{x}\\v \end{smallmatrix}\right)
      \)} (m-2-3);
    \draw[->,dotted] (m-2-3) -- node {\(
      \left(\begin{smallmatrix} p\\t \end{smallmatrix}\right)
      \)} (m-3-3);
    \draw[double, double equal sign distance] (m-3-2) -- (m-3-3);
    \draw[double, double equal sign distance] (m-1-1) -- (m-2-1);
    \draw[double, double equal sign distance] (m-1-4) -- (m-2-4);
  \end{tikzpicture}
  \]
  commutes and has exact rows and columns.
  
  We claim that \(\bar{x}=x\).  Recall that~\(F(u)\) is
  surjective, so \(F(\bar{x})\) is determined by its composite
  with~\(F(u)\), which is equal to \(F(\hat{\partial}_1)=\partial_1\)
  by the commuting diagram.  Hence \(F(\bar{x})=F(x)\), which gives
  \(\bar{x}=x\) by the uniqueness of~\(x\).  The exactness of the third
  column means that~\(A\) is the cone of the map \((x,v)\colon D
  \to\hat{Q}_0\oplus \hat{M}_1[1]\).  The canonical lifting of~\(F(A)\)
  is the cone of the map \((x,0)\colon D \to \hat{Q}_0\oplus
  \hat{M}_1[1]\).  Hence the obstruction class of~\(A\) is the image
  of \((x,v)-(x,0)=(0,v)\) under the map from \(\Ext^1_\Abel\bigl(F(D),
  Q_0[-1]\oplus M_1\bigr) \subseteq \Tri(D, \hat{Q}_0\oplus
  \hat{M}_1[1])\) to \(\Ext^2_\Abel\bigl(F(A),F(A)[-1]\bigr)\)
  constructed in the proof of Theorem~\ref{the:classify_liftings_two-dim}.  

  We may describe the element in
  \(\Ext^1_\Abel\bigl(F(D),Q_0[-1]\oplus M_1\bigr)\) induced
  by~\((0,v)\) because~\(v\) also appears in the first row: it is
  represented by the extension
  \[
  0\to Q_0[-1]\oplus M_1\xrightarrow{(\id_{Q_0[-1]},\partial_2)}
  Q_0[-1]\oplus Q_1\xrightarrow{(0,u)} F(D)\to 0.
  \]
  The next step is to push forward along the map
  \[
  (\varepsilon[-1], \id_{M_1}) \colon
  Q_0[-1]\oplus M_1\to M_0[-1]\oplus M_1\cong F(A)[-1].
  \]
  The resulting element in \(\Ext^1_\Abel\bigl(F(D),F(A)[-1]\bigr)\) is
  represented by the extension
  \[
   0\to M_0[-1]\oplus M_1\xrightarrow{(\id_{M_0[-1]},\partial_2)}
   M_0[-1]\oplus Q_1\xrightarrow{(0,u)} F(D)\to 0.
  \]

  To get the obstruction class for~\(A\) in
  \(\Ext^2_\Abel\bigl(F(A),F(A)[-1]\bigr)\), we need to splice the
  extension above with the extension
  \[
  0\to F(D)\xrightarrow{(\iota,0)} Q_0\oplus M_1[1]
  \xrightarrow{(\varepsilon,\id_{M_1[1]})} M_0\oplus M_1[1]\to 0,
  \]
  where \(\iota\colon F(D)\to Q_0\) denotes the inclusion map
  \(F(D)\cong\ker\varepsilon\subseteq Q_0\).
  Up to the identification stated in the theorem, this yields indeed the
  class of the \(2\)\nb-step extension~\eqref{eq:cone_MQQM}.

  Now we establish the isomorphism criterion.  First assume that
  there are invertible maps~\(\varphi_i\) as
  in~\eqref{eq:cone_MQQM_compare}.  Since the cone of the identity
  map is the zero object and since cones are additive for direct
  sums, the cone of \(\hat{\partial}_1\oplus \id_{\hat{Q}'_0}\) is
  again~\(A\), and the cone of \(\hat{\partial}'_1\oplus
  \id_{\hat{Q}_0}\) is again~\(A'\).  Since the maps
  \(\hat{\partial}_1\oplus \id_{\hat{Q}'_0}\) and
  \(\hat{\partial}'_1\oplus \id_{\hat{Q}_0}\) are isomorphic
  by~\eqref{eq:cone_MQQM_compare}, they have isomorphic cones.  Thus
  \(A\cong A'\).

  Conversely, assume that \(A\cong A'\).  Then \(F(A)\cong F(A')\),
  so that we get isomorphisms \(\varphi_1\colon M_1\to M_1'\) and
  \(\varphi_4\colon M_0\to M_0'\).  To simplify notation, we assume
  without loss of generality that \(\varphi_1\) and~\(\varphi_4\)
  are identity maps.  Then the isomorphism \(A\cong A'\) is an
  equivalence of liftings, so that \(A\) and~\(A'\) have the same
  obstruction class in \(\Ext^2_\Abel\bigl(F(A),F(A)[-1]\bigr)\).
  By our computation of the obstruction class, this means that the
  \(2\)\nb-step extensions \eqref{eq:cone_MQQM}
  and~\eqref{eq:cone_MQQM_prime} (with \(M_i'=M_i\)) have the same
  class in \(\Ext^2_\Abel(M_0,M_1)\).  The classes in
  \(\Ext^2_\Abel(M_0,M_1)\) are not changed by adding
  \(\id_{Q_0'}\colon Q_0'\to Q_0'\) in~\eqref{eq:cone_MQQM} and
  \(\id_{Q_0}\colon Q_0\to Q_0\) in~\eqref{eq:cone_MQQM_prime}.
  Thus the two rows in~\eqref{eq:cone_MQQM_compare} have the same
  class in \(\Ext^2_\Abel(M_0,M_1)\).

  Since \(Q_0\) and~\(Q_0'\) are projective, there are maps
  \(\psi\colon Q_0\to Q_0'\) and \(\psi'\colon Q_0'\to Q_0\) with
  \(\varepsilon'\circ \psi=\varepsilon\colon Q_0\to M\) and
  \(\varepsilon\circ \psi'=\varepsilon'\colon Q'_0\to M\).  Then
  \[
  \varphi_3\defeq
  \begin{pmatrix}
    \id_{Q_0}&0\\\psi&\id_{Q'_0}
  \end{pmatrix}\circ
  \begin{pmatrix}
    \id_{Q_0}&-\psi'\\0&\id_{Q'_0}
  \end{pmatrix}\colon
  Q_0\oplus Q_0'\to Q_0\oplus Q_0'
  \]
  is an isomorphism with \((0,\varepsilon')\circ\varphi_3 =
  (\varepsilon,0)\).  Hence~\(\varphi_3\) makes the third square
  in~\eqref{eq:cone_MQQM_compare} commute.

  Let~\(K\) be the kernel of \((\varepsilon,0)\colon Q_0\oplus Q'_0\to
  M_0\).  This is isomorphic to the kernel of \((0,\varepsilon')\)
  via~\(\varphi_3\).  Since \(Q_0\oplus Q_0'\) is projective,
  composition with the class of the extension \(K\into Q_0\oplus Q_0'
  \prto M_0\) gives an isomorphism \(\Ext^2_\Abel(M_0,M_1) \cong
  \Ext^1_\Abel(K,M_1)\).  Hence the equality of the obstruction
  classes shows that the extensions \(M_1\into Q_1\oplus Q_0'\prto K\)
  and \(M_1\into Q_0\oplus Q_1'\prto K\) that we get from the two rows
  in~\eqref{eq:cone_MQQM_compare} and the isomorphism~\(\varphi_3\)
  have the same class in \(\Ext^1_\Abel(K,M_1)\).  Equality in
  \(\Ext^1_\Abel(K,M_1)\) means that the extensions really are
  isomorphic in the strongest possible sense, that is, there is an
  isomorphism \(\varphi_2\colon Q_1\oplus Q_0' \to Q_0\oplus Q'_1\)
  that induces an isomorphism of extensions.  This means that it makes
  the remaining two squares in~\eqref{eq:cone_MQQM_compare} commute.
  Thus \(A\cong A'\) implies that there are isomorphisms~\(\varphi_i\)
  making~\eqref{eq:cone_MQQM_compare} commute.
\end{proof}

\begin{remark}
  \label{rem:MQQM_drop_parity}
  The same argument works if the objects in~\eqref{eq:cone_MQQM} all
  belong to~\(\Abel_-\).  If the objects in~\eqref{eq:cone_MQQM}
  belong to~\(\Abel\), then we may split~\eqref{eq:cone_MQQM} into its
  even and odd parts.  Thus the obvious adaption of
  Theorem~\ref{the:obstruction_class_cone} still holds without any
  parity assumptions on the objects \(M_j\) and~\(Q_j\).
\end{remark}

\begin{remark}
  \label{rem:MQQM_other_criterion}  
  There are several variants of the
  criterion~\eqref{eq:cone_MQQM_compare}.  Since \eqref{eq:cone_MQQM}
  and~\eqref{eq:cone_MQQM_prime} are exact, isomorphisms \(\varphi_2\)
  and~\(\varphi_3\) making the middle square
  in~\eqref{eq:cone_MQQM_compare} commute give \(\varphi_1\)
  and~\(\varphi_4\) making all squares in~\eqref{eq:cone_MQQM_compare}
  commute.  Furthermore, if~\(\varphi_i\) are isomorphisms for \(i=1,4\)
  and for \(i=2\) or \(i=3\), then the remaining one is an isomorphism
  as well by the Five Lemma.

  If there are maps~\(\varphi_i\) making~\eqref{eq:cone_MQQM_compare}
  commute, and such that \(\varphi_1\) and~\(\varphi_4\) are
  invertible, then it already follows that \(A\cong A'\).  This is
  because \(\varphi_1\) and~\(\varphi_4\) induce an isomorphism
  \(F(A)\cong F(A')\), and~\eqref{eq:cone_MQQM_compare} shows that the
  obstruction classes also agree, no matter whether \(\varphi_2\)
  or~\(\varphi_3\) are invertible.
\end{remark}

\subsection{Crossed products for \Cstar{}algebras over topological spaces}
\label{sec:crossed_products}

In this subsection, we generalise some basic results about crossed
products to \(\Cst\)\nb-algebras over topological spaces.  Let~\(G\)
be a locally compact group.  Let~\(X\) be a second countable
topological space.

\begin{definition}
  A \emph{\(G\)\nb-\(\Cst\)\nb-algebra over~\(X\)} is a
  \(\Cst\)\nb-algebra over~\(X\) whose underlying
  \(\Cst\)\nb-algebra is a \(G\)\nb-\(\Cst\)\nb-algebra such that
  all distinguished ideals are \(G\)\nb-invariant.
\end{definition}

If \((A,\alpha)\) is a \(G\)\nb-\(\Cst\)\nb-algebra over~\(X\) then
the crossed product \(A\rtimes_\alpha G\) is a \(\Cst\)\nb-algebra
over~\(X\) via \((A\rtimes_\alpha G)(U)\defeq
A(U)\rtimes_{\alpha|_{A(U)}} G\) for all \(U\in\mathbb{O}(X)\).
If~\(G\) is Abelian, then \((A\rtimes_\alpha G,\hat\alpha)\) is a
\(\hat G\)\nb-\(\Cst\)\nb-algebra over~\(X\).

\begin{proposition}[Takai Duality]
  Let~\(G\) be Abelian.  Let \((A,\alpha)\) be a
  \(G\)\nb-\(\Cst\)\nb-algebra over~\(X\).  Then there is a natural
  isomorphism of \(G\)\nb-\(\Cst\)\nb-algebras over~\(X\) from
  \(\bigl((A\rtimes_\alpha G)\rtimes_{\hat\alpha} \hat
  G,\hat{\hat\alpha}\bigr)\) to \(\bigl(A\otimes\mathbb K(L^2 G),
  \alpha\otimes\mathrm{ad}_\lambda\bigr)\).  In particular, there is
  a natural \(\KK(X)\)-equivalence between \((A\rtimes_\alpha G)
  \rtimes_{\hat\alpha} \hat G\) and~\(A\).
\end{proposition}

\begin{proof}
  We have only added \(X\)-equivariance to the classical statement.
  This follows immediately from the naturality of the classical
  version (see \cite{Connes:NCG}*{Theorem~6 in Appendix~C of
    Chapter~2}).
\end{proof}

In the following, we assume for convenience that \(X\) is finite and
\(A\) belongs to the category \(\KKcat(X)_\loc\) defined in
\cite{Meyer-Nest:Bootstrap}*{Definition 4.8}.  It should be possible
to remove these assumptions by carefully checking the naturality of the
homotopies effecting the respective equivalences.

\begin{proposition}[Green Imprimitivity]
  Let \(H\) be a closed subgroup of~\(G\) and let \((A,\alpha)\) be
  an \(H\)\nb-\(\Cst\)\nb-algebra over~\(X\).  Assume that \(X\) is
  finite and \(A\inOb\KKcat(X)_\loc\).  There is a natural
  \(\KK(X)\)-equivalence between \(A\rtimes_\alpha H\)
  and~\(\Ind_H^G(A,\alpha)\rtimes_{\Ind\alpha}G\).
\end{proposition}

\begin{proof}
  By naturality, the imprimitivity bimodule constructed in
  \cite{Echterhoff-Kaliszewski-Quigg-Raeburn:Naturality}*{Theorem
    4.1} induces a \(\KK(X)\)-element which is a pointwise
  \(\KK\)-equivalence.  By \cite{Meyer-Nest:Bootstrap}*{Proposition
    4.9}, it is a \(\KK(X)\)-equivalence.
\end{proof}

\begin{proposition}[Connes--Thom Isomorphism]
  Let \((A,\alpha)\) be an \(\R\)\nb-\(\Cst\)\nb-algebra over~\(X\).
  Assume that \(X\) is finite and \(A\inOb\KKcat(X)_\loc\).  Then
  \(A\rtimes_\alpha \R\) is naturally \(\KK(X)\)-equivalent to
  \(A[-1]\).
\end{proposition}

\begin{proof}
  We adopt the approach from \cite{Koehler:Thesis}*{Proposition
    8.3}.  Let \(\tilde{A}\) denote the \(\Cst\)\nb-al\-ge\-bra \(A\)
  over~\(X\) with the trivial \(\R\)\nb-action.  Then
  \(\CONT_0(\R,A)\) and \(\CONT_0(\R,\tilde{A})\) with the diagonal
  actions are naturally \Star{}isomorphic as
  \(\R\)\nb-\(\Cst\)\nb-algebras.  By
  \cite{Kasparov:Novikov}*{Theorem 5.9} we may fix a
  \(\KK^\R\)-equivalence between \(\C[-1]\) and \(\CONT_0(\R)\),
  where \(\CONT_0(\R)\) carries the translation action.  In
  combination, this gives a natural \(\KK^\R\)-equivalence between
  \(A[-1]\) and \(\tilde{A}[-1]\), and consequently also between
  \(A\) and \(\tilde{A}\).  Taking crossed products gives a natural
  \(\KK\)-equivalence between \(A\rtimes_\alpha\R\) and \(A[-1]\).
  As in the previous proof, the naturality of the constructed cycle
  and \cite{Meyer-Nest:Bootstrap}*{Proposition 4.9} show that this
  is a \(\KK(X)\)-equivalence.
\end{proof}

\begin{proposition}[Pimsner--Voiculescu Triangle]
  Let \((A,\alpha)\) be a \(\Z\)\nb-\(\Cst\)\nb-algebra over~\(X\).
  Assume that~\(X\) is finite and \(A\inOb\KKcat(X)_\loc\).  Then
  there is a natural exact triangle in \(\KKcat(X)\) of the form
  \[
  A[-1]\rtimes_\alpha\Z\to A\xrightarrow{\alpha(1)-\mathrm{id}} A
  \to A\rtimes_\alpha\Z.
  \]
\end{proposition}

\begin{proof}
  We abbreviate \(\alpha=\alpha(1)\) and let
  \(T_\alpha=\bigl\{f\colon\CONT(\left[0,1\right],A)\mid
  f(1)=\alpha\bigl(f(0)\bigr)\bigr\}\) be the mapping torus
  of~\(\alpha\).  The extension \[A[-1]\rightarrowtail
  T_\alpha\mathop{\stackrel{\ev_0} {\twoheadrightarrow}} A\] has a
  completely positive \(X\)\nb-equivariant section taking \(a\in A\)
  to the affine function \((1-t)\cdot a +t\cdot\alpha(a)\).  We get
  a natural exact triangle in \(\KKcat(X)\) of the form
  \begin{equation}
    \label{eq:exact_triangle_for_PV}
    A[-1]\to A[-1]\to T_\alpha\to A.
  \end{equation}
  The \(\R\)\nb-\Cstar{}algebra \(T_\alpha\) is naturally
  \Star{}isomorphic over~\(X\) to \(\Ind_\Z^\R(A,\alpha)\).  By
  Green's imprimitivity theorem and the Connes--Thom isomorphism, we
  have natural \(\KK(X)\)-equivalences
  \[
  A\rtimes_\alpha\Z\simeq\Ind_\Z^\R(A,\alpha)\rtimes_{\Ind\alpha}\R
  \simeq T_\alpha\rtimes\R\simeq T_\alpha[-1].
  \]
  Plugging this into~\eqref{eq:exact_triangle_for_PV} and rotating
  as appropriate gives an exact triangle of the desired form.  The
  formula for the map from~\(A\) to~\(A\) is a consequence of the
  naturality of the boundary map in the \(\KK\)\nb-theoretic
  six-term sequence applied to the morphisms of extensions
  \[
  \begin{tikzpicture}[baseline=(current bounding box.west)]
    \matrix(m)[cd,column sep=1em]{
      A[-1]&T_\alpha&A\\
      A[-1]&\CONT\bigl(\left[0,1\right],A\bigr)&A\oplus A,\\
    };
    \begin{scope}[cdar]
      \draw[>->] (m-1-1) -- (m-1-2);
      \draw[->>] (m-1-2) -- (m-1-3);
      \draw[>->] (m-2-1) -- (m-2-2);
      \draw[->>] (m-2-2) -- (m-2-3);
      \draw[>->] (m-1-2) -- (m-2-2);
      \draw (m-1-3) -- node {\((\id_A,\alpha)\)}  (m-2-3);
    \end{scope}
    \draw[double, double equal sign distance] (m-1-1) --   (m-2-1);
  \end{tikzpicture}\quad
  \begin{tikzpicture}[baseline=(current bounding box.west)]
    \matrix(m)[cd,row sep=12pt, column sep=1em]{
      A[-1]&\CONT_0\bigl(\left[0,1\right),A\bigr)& A\\
      A[-1]&\CONT\bigl(\left[0,1\right],A\bigr)&A\oplus A\\
      A[-1]&\CONT_0\bigl(\left(0,1\right],A\bigr)& A,\\
    };
    \begin{scope}[cdar]
      \draw[>->] (m-1-1) -- (m-1-2);
      \draw[->>] (m-1-2) -- (m-1-3);
      \draw[>->] (m-2-1) -- (m-2-2);
      \draw[->>] (m-2-2) -- (m-2-3);
      \draw[>->] (m-3-1) -- (m-3-2);
      \draw[->>] (m-3-2) -- (m-3-3);
      \draw[>->] (m-1-2) -- (m-2-2);
      \draw[>->] (m-3-2) -- (m-2-2);
      \draw (m-1-3) -- node {\(\iota_0\)}  (m-2-3);
      \draw (m-3-3) -- node[swap] {\(\iota_1\)}  (m-2-3);
    \end{scope}
    \draw[double, double equal sign distance] (m-1-1) --   (m-2-1);
    \draw[double, double equal sign distance] (m-3-1) --   (m-2-1);
  \end{tikzpicture}
  \]
  together with the elementary fact that the extensions
  \(\C[-1]\rightarrowtail\CONT_0\bigl(\left[0,1\right)\bigr)
  \twoheadrightarrow\C\) and \(\C[-1]\rightarrowtail
  \CONT_0\bigl(\left(0,1\right]\bigr)\twoheadrightarrow\C\)
  correspond to the classes \(-\id_{\C[-1]}\) and \(\id_{\C[-1]}\)
  in \(\KK_1(\C,\C[-1])\cong \KK_0(\C[-1],\C[-1])\), respectively.
\end{proof}

\begin{corollary}[Dual Pimsner--Voiculescu Triangle]
  \label{cor:dual_PV}
  Let \((A,\alpha)\) be a \(\T\)\nb-\(\Cst\)\nb-algebra over~\(X\).
  Assume that \(X\) is finite and
  \(A\rtimes_\alpha\T\inOb\KKcat(X)_\loc\).  Then there is a natural
  exact triangle in \(\KKcat(X)\) of the form
 \[
  A[-1]\to A\rtimes_\alpha\T\xrightarrow{\hat\alpha(1)-\mathrm{id}}
  A\rtimes_\alpha\T\to A.
 \]
\end{corollary}

\begin{proof}
  This follows from the Pimsner--Voiculescu Triangle and Takai
  Duality.
\end{proof}

\subsection{Application to graph algebras}
\label{sec:graph_apply}

Let \(A=\Cst(E)\) be the \(\Cst\)\nb-algebra of a countable
graph~\(E\).  We assume that~\(A\) has only finitely many ideals or,
equivalently, that its primitive ideal space is finite; this is
necessary for our machinery to work.  We set \(X=\Prim(A)\).  The
gauge action \(\gamma\colon\T\curvearrowright A\) turns~\(A\) into a
\(\T\)\nb-\(\Cst\)\nb-algebra over~\(X\).  Corollary~\ref{cor:dual_PV}
provides the following natural exact triangle in~\(\KKcat(X)\):
\[
A[-1]\to A\rtimes_\gamma\T\xrightarrow{\hat\gamma(1)^{-1}-\mathrm{id}} 
A\rtimes_\gamma\T\to A
\]

The \(\Cst\)\nb-algebra~\(A\rtimes_\gamma\T\) is AF.  Hence the odd
part of \(\XK(A\rtimes_\gamma\T)\) vanishes.  Applying the
functor~\(\XK\) to the dual Pimsner--Voiculescu triangle, we get the
following dual Pimsner--Voiculescu exact sequence:
\begin{equation}
  \label{eq:PV_ses}
 0\to\XK_1(A)\to\XK_0(A\rtimes_\gamma\T)
 \xrightarrow{\hat\gamma(1)_*^{-1}-\mathrm{id}}
 \XK_0(A\rtimes_\gamma\T)\to\XK_0(A)\to 0.
\end{equation}
The module \(\XK(A\rtimes_\gamma\T)\) is usually not projective, so we
cannot directly apply Theorem~\ref{the:obstruction_class_cone}.  For
this purpose, we replace \(A\rtimes_\gamma\T\) by a suitable
\(\Cst\)\nb-subalgebra.  This construction is based on the ingredients
of the computation of the \(\K\)\nb-theory of graph
\(\Cst\)\nb-algebras in
\cite{Raeburn-Szymanski:CK_algs_of_inf_graphs_and_matrices}*{Section~3}
and \cite{Bates-Hong-Raeburn-Szymanski:Ideal_structure}*{Section~6};
we shall use the notation and a number of results proved in these
articles.

We may identify \(\Cst(E)\rtimes_\gamma\T\) with the \(\Cst\)\nb-algebra
of the so-called skew-product graph \(E\times_1\Z\).  This becomes an
isomorphism of \(\Cst\)\nb-algebras over~\(X\) via the canonical definitions
\(\bigl(\Cst(E)\rtimes_\gamma\T\bigr)(U)\defeq\bigl(\Cst(E)(U)\bigr)
\rtimes_\gamma\T\) and \(\Cst(E\times_1\Z)(U)\defeq J_{H_U\times\Z,B_U
\times\Z}\) where \((H_U,B_U)\) is the admissible pair such that
\(\Cst(E)(U)=J_{H_U,B_U}\).

We let \(N\) denote the set \(\{n\in\Z\mid n
\leq 0\}\) and \(N^*=N\setminus\{0\}\).  We let \(\hat{Q}_0=\hat{Q}_1\)
be the \(\Cst\)\nb-subalgebra of \(\Cst(E\times_1\Z)\) associated to the
subgraph \(E\times_1 N\), the restriction of the graph \(E\times_1\Z\) to the
subset of vertices \(E\times N\).  This is a \(\Cst\)\nb-algebra over~\(X\)
via \(\Cst(E\times_1 N)(U)=J_{H_U\times N, B_U\times N^*}\), and the
inclusion map \(\Cst(E\times_1 N)\hookrightarrow\Cst(E\times_1\Z)\) is a
\Star{}homomorphism over~\(X\).

\begin{lemma}
  \label{lem:graph_Q_i_projective}
  The module~\(\XK(\hat{Q}_0)\) is projective and concentrated in even degree.
\end{lemma}

\begin{proof}
  The \(\Cst\)\nb-algebra~\(\hat{Q}_0\) is an AF-algebra because the
  graph~\(E\times_1 N\) has no cycle.  We claim that all distinguished
  subquotients of~\(\hat{Q}_0\) have free \(\K_0\)\nb-groups and
  vanishing \(\K_1\)\nb-groups.  Since ideals and quotients of
  AF-algebras are again AF and since AF-algebras have vanishing
  \(\K_1\)\nb-groups, it suffices to show that the \(\K_0\)\nb-group
  of every distinguished quotient of \(\Cst(E\times_1 N)\) is free.
  By
  \cite{Bates-Hong-Raeburn-Szymanski:Ideal_structure}*{Corollary~3.5},
  the quotient of \(\Cst(E\times_1 N)\) by the ideal \(J_{H_U\times_1
    N, B_U\times_1 N^*}\) is isomorphic to the \(\Cst\)\nb-algebra of
  the graph \(\bigl((E\times_1 N) / (H\times_1 N)\bigr) \setminus
  \beta(B\times_1 N^*)\); its \(\K\)\nb-theory is free by
  \cite{Bates-Hong-Raeburn-Szymanski:Ideal_structure}*{Lemma~6.2} and
  continuity of \(\K\)\nb-theory.  Now it follows from
  \cite{Bentmann:Real_rank_zero_and_int_cancellation}*{Lemma~4.10}
  that the module~\(\XK(\hat{Q}_0)\) is projective.
\end{proof}
  
The dual action on \(A\rtimes_\gamma\T\) corresponds to the shift
automorphism~\(\beta\) on \(\Cst(E\times_1\Z)\) whose inverse preserves
the subalgebra \(\hat{Q}_0=\Cst(E\times_1 N)\).  We
write \(s\colon \hat{Q}_1\to \hat{Q}_0\) for the restricted morphism
\(1-\beta^{-1}\) and form an exact triangle
\[
  C_s\to \hat{Q}_1\xrightarrow{s} \hat{Q}_0\to  C_s[1].
\]
By an axiom of triangulated categories, there is a morphism \(f\colon
C_s\to A[-1]\) such that the diagram
\begin{equation}
\label{eq:morphism_of_triangles}
  \begin{tikzpicture}[baseline=(current bounding box.west)]
    \matrix(m)[cd,column sep=4em]{
      C_s&\hat{Q}_1&\hat{Q}_0&C_s[1]\\
      A[-1]&A\rtimes_\gamma\T&A\rtimes_\gamma\T&A\\
    };
    \begin{scope}[cdar]
      \draw (m-1-1) -- (m-1-2);
      \draw (m-1-2) -- node {\(1-\beta^{-1}\)} (m-1-3);
      \draw (m-1-3) -- (m-1-4);
      \draw (m-2-1) -- (m-2-2);
      \draw (m-2-2) -- node {\(1-\hat\gamma^{-1}\)} (m-2-3);
      \draw (m-2-3) -- (m-2-4);
      \draw (m-1-1) -- node[swap] {\(f\)} (m-2-1);
      \draw[>->] (m-1-2) -- (m-2-2);
      \draw[>->] (m-1-3) -- (m-2-3);
      \draw (m-1-4) -- node {\(f[1]\)} (m-2-4);
    \end{scope}
  \end{tikzpicture}
\end{equation}
commutes.  As in
\cite{Raeburn-Szymanski:CK_algs_of_inf_graphs_and_matrices}*{Lemma~3.3},
it follows that the morphism \(f_*\colon\K_*\bigl(C_s(Y)\bigr)\to
\K_*\bigl(A(Y)\bigr)\) induced by~\(f\) is bijective for every closed
subset \(Y\subseteq X\).  Hence, by
\cite{Meyer-Nest:Bootstrap}*{Proposition~4.15} and the Five Lemma, it is
a \(\KK(X)\)-equivalence.

\begin{lemma}
  \label{lem:graph_kernel_K1_projective}
  The module \(\ker\bigl(\XK(s)\bigr)\) is projective.
\end{lemma}

\begin{proof}
  This module is concentrated in even degree and isomorphic to
  \(\XK_1\bigl(\Cst(E)\bigr)\).  Since~\(\Cst(E)\) has vanishing
  exponential maps and all its subquotients have free
  \(\K_1\)\nb-groups, it follows as in
  \cite{Bentmann:Real_rank_zero_and_int_cancellation}*{Lemma~4.10}
  that this module is projective.
\end{proof}

Since the map~\(f\) in~\eqref{eq:morphism_of_triangles} is a
\(\KK(X)\)-equivalence, we get an exact triangle
\[
A[-1]\to \hat{Q}_1\to \hat{Q}_0\to A.
\]
Lemmas~\ref{lem:graph_Q_i_projective}
and~\ref{lem:graph_kernel_K1_projective} show that
Theorem~\ref{the:obstruction_class_cone} applies.  The obstruction
class is the image of the top row in~\eqref{eq:morphism_of_triangles}
under the functor~\(\XK\).  The vertical maps in this diagram show
that~\(\XK\) applied to the bottom row also represents the same class
in~\(\Ext^2\).  The bottom row is exactly the dual Pimsner--Voiculescu
sequence~\eqref{eq:PV_ses}, as asserted.  Combining these computations
with Kirchberg's Classification Theorem gives the following theorem:

\begin{theorem}
  \label{thm:graph_alg_classification}
  Let \(A_1\) and~\(A_2\) be purely infinite graph
  \(\Cst\)\nb-algebras such that \(\Prim(A_1)\cong\Prim(A_2)\cong X\)
  is finite.  Then any isomorphism
  \(\XK\delta(A_1)\cong\XK\delta(A_2)\) lifts to a stable isomorphism
  between \(A_1\) and~\(A_2\).  The obstruction classes
  \(\delta(A_i)\) in \(\Ext^2\bigl(\XK(A_i),\XK(A_i)[-1]\bigr)\) are
  determined by the dual Pimsner--Voiculescu
  sequences~\eqref{eq:PV_ses} for the gauge actions
  \(\gamma\colon\T\curvearrowright A_i\).
\end{theorem}

\begin{proof}
  By \cite{Kirchberg-Rordam:Infinite_absorbing}*{Corollary 9.4}, a
  purely infinite, separable, nuclear \Cstar{}algebra with real rank
  zero absorbs the infinite Cuntz algebra \(\mathcal{O}_\infty\)
  tensorially.  Hence Kirchberg's Classification Theorem applies.
  It gives the result together with
  Theorem~\ref{the:obstruction_class_cone}.
\end{proof}

Roughly speaking, stable, purely infinite graph \(\Cst\)\nb-algebras
with finitely many ideals are strongly classified by their dual
Pimsner--Voiculescu sequence in~\(\XK\) (up to the correct notion of
equivalence).

\begin{corollary}
  \label{cor:unital_graph_alg_classification}
  Let \(A_1\) and \(A_2\) be unital, purely infinite graph
  \(\Cst\)\nb-algebras such that \(\Prim(A_1)\cong\Prim(A_2)\cong X\).
  Then any isomorphism \(\XK\delta(A_1)\cong\XK\delta(A_2)\)
  taking the unit class in~\(\K_0(A_1)\) to the unit class
  in~\(\K_0(A_2)\) lifts to a \Star{}isomorphism between \(A_1\)
  and~\(A_2\).
\end{corollary}

\begin{proof}
  This follows from
  \cite{Eilers-Restorff-Ruiz:Strong_class_of_ext}*{Theorem~3.3} and
  our strong classification theorem up to stable isomorphism.  Here we
  use that, when~\(A\) has real rank zero, the group~\(\K_0(A)\) can
  be naturally recovered from the module~\(\XK_0(A)\) as a certain
  cokernel, see \cite{Arklint-Bentmann-Katsura:Reduction}*{Lemma~8.3}.
\end{proof}

\end{document}